\documentclass[10pt,reqno]{amsart}
\usepackage{amsmath,amsthm,amsfonts,amssymb,amscd}
\usepackage{mathtools}
\usepackage[noadjust]{cite}
\usepackage{graphicx}
\usepackage{verbatim,enumerate}
\usepackage[active]{srcltx}
\usepackage{geometry}
\usepackage{times}
\usepackage{hyperref}
 \newtheorem{theorem}{Theorem}[section]
 \newtheorem*{theorem*}{Theorem}
 \newtheorem*{lemma*}{Lemma}
 \newtheorem{proposition}[theorem]{Proposition}
 
 \newtheorem{fact*}{Fact}
 \newtheorem{lemma}[theorem]{Lemma}
 \newtheorem{corollary}[theorem]{Corollary}
\theoremstyle{definition}

 \newtheorem*{remark*}{Remark}
 
 \newtheorem{example}[theorem]{Example}
\numberwithin{equation}{section}
\renewcommand{\theenumi}{{\rm(\arabic{enumi})}}

\newcommand{\vect}[1]{\boldsymbol{#1}}
\newcommand{\R}{\boldsymbol{R}}
\newcommand{\A}{\mathcal{A}}

\newcommand{\rank}{\operatorname{rank}}

\newcommand{\ord}{\operatorname{ord}}
\renewcommand{\phi}{\varphi}

\newcommand{\sgn}{\operatorname{sgn}}
\newcommand{\inner}[2]{\left\langle{#1},{#2}\right\rangle}

\newcommand{\e}{\vect{e}}
\newcommand{\n}{\vect{n}}
\newcommand{\pmt}[1]{{\begin{pmatrix} #1  \end{pmatrix}}}
\newcommand{\vmt}[1]{{\begin{vmatrix} #1  \end{vmatrix}}}
\newcommand{\M}{\mathcal{M}}
\newcommand{\what}{\widehat}
\allowdisplaybreaks[2]

\title[Boundedness of geometric invariants near a singularity]
{Boundedness of geometric invariants near a singularity
which is a suspension of a singular curve}
\dedicatory{Dedicated to Professors Masaaki Umehara and Kotaro Yamada
on the occasions of their sixtieth birthdays}
\author[L. F. Martins]{Luciana F. Martins} 
\author[K. Saji]{Kentaro Saji} 
\author[S. P. dos Santos]{Samuel P. dos Santos} 
\author[K. Teramoto]{Keisuke Teramoto}
\address[L. F. Martins \& S. P. dos Santos]
{
Universidade Estadual Paulista (UNESP), 
Instituto de Bio{\-}ci\^{e}n{\-}cias Letras e Ci\^{e}ncias Exatas,
R.~Crist{\'o}v{\~a}o Colombo, 2265, Jd Nazareth, 
15054-000 S{\~a}o Jos{\'e} do Rio Preto, S\~{a}o Paulo, Brazil}
\email{luciana.martins@unesp.br}
\email{samuel.paulino@unesp.br}
\address[K. Saji]{Department of Mathematics, 
Graduate School of Science, 
Kobe University, 
Rokkodai 1-1, Nada, Kobe, 657-8501, Japan}
\email{saji@math.kobe-u.ac.jp}
\address[K. Teramoto]{
Graduate School of Sciences and Technology for Innovation,
Yamaguchi University, Yamaguchi, 753-8512, Japan}
\email{kteramoto@yamaguchi-u.ac.jp}
\date{\today}
\thanks{Partly supported by the
Japan Society for the Promotion of Science KAKENHI 
Grants numbered 18K03301, 19K14533, 22K03312, 22K13914,
the Japan-Brazil bilateral project JPJSBP1 20190103
and
the S\~ao Paulo Research Foundation grant 2018/17712-7.}
\begin{document}
%\dedicatory{Dedicated to Professors Masaaki Umehara 
%and Kotaro Yamada
%on the occasions of their sixtieth birthdays}
\begin{abstract}
Near a singular point of a surface or a curve, geometric invariants diverge in general, and the orders of diverge, in particular the boundedness about these invariants represent geometry of the surface and the curve.
In this paper, we study boundedness and orders of several geometric invariants near a singular point of a surface which is a suspension of a singular curve in the plane
and those of curves passing through the singular point.
We evaluates the orders of Gaussian and mean curvatures
and them of geodesic, normal curvatures and geodesic torsion for the curve.
\end{abstract}

\maketitle

%%%%%%%%%%%%%%%%%%%%%%%%%%%%%%%%%%%%%%%%%%%%%%%%%%
%%%%%%%%%%%%%%%%%%%%%%%%%%%%%%%%%%%%%%%%%%%%%%%%%%
\section{Introduction}
\label{sec:intro}
In this paper, we study boundedness of 
several geometric invariants near a 
singular point of a surface
which is a suspension of a singular curve in the plane.
More precisely, let $\sigma$ be an $\A$-equivalence class
of singular plane curve-germs.
A {\it $\sigma$-edge\/} is a map-germ
$f:(\R^2,0)\to(\R^3,0)$ such that it is $\A$-equivalent
to $(u,v)\mapsto(u,c_1(v),c_2(v))$,
where $c=(c_1,c_2)$ is a representative of $\sigma$,
namely, a one-dimensional {\it suspension\/} of $\sigma$.
Here, two map-germs
$h_1,h_2:(\R^m,0)\to(\R^n,0)$ are {\it $\A$-equivalent\/}
if there exist diffeomorphisms
$\Phi_s:(\R^m,0)\to(\R^m,0)$ and $\Phi_t:(\R^n,0)\to(\R^n,0)$
such that $h_2=\Phi_t\circ h_1\circ \Phi_s^{-1}$.
A cuspidal edge 
($\mathcal{A}$-equivalent to the germ $(u,v)\mapsto(u,v^2,v^3)$ 
at the origin)
and a $5/2$-cuspidal edge 
($\mathcal{A}$-equivalent to the germ $(u,v)\mapsto(u,v^2,v^5)$)
are examples of $\sigma$-edges,
and $\sigma$ are a $3/2$-cusp and $5/2$-cusp respectively.
If $\sigma$ is of finite multiplicity, then the $\sigma$-edge is a frontal.
A frontal is a class of surfaces with singular points,
and it is well known that surfaces with 
constant curvature are frequently in this class.
In these decades,
there are several studies of frontals 
from the viewpoint of differential geometry 
and various geometric invariants at singular points 
are introduced (for instance 
\cite{maxface,framed,intrinsic-frontal,honda-saji,martins-saji,msuy,front}). 
If a surface is invariant under a group action on $\R^3$, then
$\sigma$-edges will appear naturally. Singularities appearing on 
surfaces of revolution and 
a helicoidal surface are
examples of such surfaces \cite{msst,tt}.
Moreover, such singularities appear
on the dual surface at cone like singular points 
of a constant mean curvature one surface
in the de Sitter $3$-space (see \cite{honda-sato}).

In this paper, we study geometry of $\sigma$-edges. 
For this, we consider two classes of singular map-germs, 
that we shall call \textit{$m$-type} and \textit{$(m,n)$-type edges}, 
the first  including $(m,n)$-type edges and also $\sigma$-edges when $\sigma$ has finite multiplicity 
(see Section \ref{sec:sigmaedge}). 
One observes that $m$-type edges are frontals. 
In order to proceed our study we find a normal form 
for each one of these map-germs preserving the geometry of the initial map, 
since we only use isometries in the target (Proposition \ref{prop:edgenormal}). 
In \cite{martins-saji,front} the authors define singular, normal and cuspidal curvatures, 
as well as cuspidal torsion for frontals. 
In an analogous way, we define similar geometric invariants  for $m$-type edges, 
getting same names, except for the cuspidal curvature, which we called $(m,m+i)$-cuspidal curvature. 
These  invariants are related with the coefficients of the normal form given in Proposition \ref{prop:edgenormal}. 
It is worth mention that these cuspidal curvatures are similar. 
In fact, we know that a frontal-germ is a front if and only if the cuspidal curvature is not zero. 
We conclude from  Proposition \ref{prop:type:front} that a $m$-type edge  is a front 
if and only if the $(m,m+1)$-cuspidal curvature is non zero at 0.
In particular, we study orders of geometric invariants
and geometric invariants of curves passing through the singular point.
We evaluates the orders of 
Gaussian and mean curvatures (Theorem \ref{thm:ordkh})
and the minimum orders of geodesic, normal curvatures and geodesic torsion
for a singular curve 
passing through the singular point (Theorem \ref{thm:ord1}).
These minimum orders are written in terms of singular, cuspidal 
and normal curvatures and the cuspidal torsion.
As a corollary, we give the boundedness of these curvatures under certain
generic conditions (Corollary \ref{cor:bdd}). 

\section{Geometry of $\sigma$-edges}\label{sec:sigmaedge}
We give several classes similar to $\sigma$-edges.
It includes $\sigma$-edges, and these classes will be
useful to treat.
We recall that a map-germ $f:(\R^2,0)\to(\R^3,0)$ is a {\it frontal\/} if
there exists a unit vector field $\nu$ along $f$
such that $\inner{df_p(X_p)}{\nu(p)}=0$ holds at any $p\in (\R^2,0)$
and any $X_p\in T_p\R^2$, where $\inner{\cdot}{\cdot}$ is the canonical inner product of $\R^3$.
The vector field $\nu$ is called a 
{\it unit normal vector field\/} of $f$. 
A map-germ $f:(\R^2,0)\to(\R^3,0)$ is an $m$-{\it type edge\/} if 
it is $\A$-equivalent
to $(u,v^m,v^{m+1}a(u,v))$ for a function $a(u,v)$.
A map-germ $f:(\R^2,0)\to(\R^3,0)$ is a $(m,n)$-{\it type edge\/}
($m<n$)
if it is $\A$-equivalent to 
$(u,v^m,v^n h(u,v))$, where $h(0,0)=1$.
This is equivalent to being $\A^n$-equivalent to 
$(u,v^m,v^n)$.
Two map-germs are $\A^n$-equivalent if their $n$-jets at the
origin are $\A$-equivalent.
We show:
\begin{lemma}\label{lem:edgem}
Let $f$ be a $\sigma$-edge $($respectively, an $m$-type edge, 
an $(m,n)$-type edge$)$.
Then an intersection curve of
$f$ with a surface $T$ which is transversal to $f(S(f))$ 
passing through $p\in S(f)$ near $0$
is 
$\A$-equivalent to $\sigma$
$($respectively, 
$\A^m$-equivalent to $(t^m,0)$,
$\A^n$-equivalent to $(t^m,t^n))$.
\end{lemma}
\begin{proof}
Since the assumption and the assertion do not depend on the
choice of the coordinate systems, 
we can assume $f$ is given by $(u,c_1(v),c_2(v))$, 
where $c=(c_1,c_2)$ is $\A$-equivalent to $\sigma$.
Then $T$ can be represented by the graph $\{(x,y,z)\,|\,x=h(y,z)\}$
in $(\R^3,0)$ as the $xyz$-space, and
the intersection curve is $(h(c_1(v),c_2(v)),c_1(v),c_2(v))$.
Since $T$ is transverse to the $x$-axis, 
the orthogonal projection of $T$ to the $yz$-plane is a diffeomorphism,
we see the assertion.
One can show the other claims by the similar way.
\end{proof}

\subsection{A sufficient condition}
We give a sufficient condition of a frontal-germ being an $m$ or
$(m,n)$-type edge
under the assumption $n<2m$.
We assume $n < 2m$ throughout this subsection.
Let $f:(\R^2,0)\to(\R^3,0)$ be a frontal-germ
satisfying $\rank df_0=1$. Then there exists a vector field
$\eta$ such that $\eta_p$ generates $\ker df_p$ if $p\in S(f)$.
We call $\eta|_{S(f)}$ a {\it null vector field\/}, and
$\eta$ an {\it extended null vector field\/}.
An extended null vector field is also called a null vector field
if it does not induce a confusion.
We assume that the set of singular points $S(f)$ is a regular curve,
and the tangent direction of $S(f)$ is not in $\ker df_0$.
Let $\xi$ be a vector field such that $\xi_p$ is a non-zero
tangent vector of $S(f)$ for $p\in S(f)$.
We consider the following conditions for $(\xi,\eta)$:
\renewcommand*{\theenumi}{{\rm [\thesection.\arabic{enumi}]}}
\begin{enumerate}
\item\label{itm:cri1}
 $\eta^{i}f=0$ $(1\leq i\leq m-1)$ on $S(f)$,
\item\label{itm:cri2}
$\rank(\xi f,\eta^{m}f)=2$ on $S(f)$,
\item\label{itm:cri3} 
$\rank(\xi f,\eta^{m}f,\eta^{i}f)=2$
$(m<i<n)$ on $S(f)$,
\item\label{itm:cri4} $\rank(\xi{f},\eta^{m}f,\eta^{n}f)=3$ at $p$.
\end{enumerate}
\renewcommand*{\theenumi}{{\rm (\arabic{enumi})}}
Here, 
for a vector field $\zeta$ and a map $f$,
the symbol
$\zeta^if$ stands for the $i$-times directional derivative
of $f$ by $\zeta$.
Moreover, for a coordinate system $(u,v)$ and a map $f$,
the symbol $f_{v^i}$
stands for $\partial^if/\partial v^i$.
\begin{proposition}\label{prop:sufcond}
Let $f:(\R^2,0)\to(\R^3,0)$ be a frontal-germ
satisfying $\rank df_0=1$.
Assume that the set of singular points $S(f)$ is a regular curve,
and the tangent direction of $S(f)$ generated by $\xi$ is not in $\ker df_0$.
If there exists a null vector field $\eta$ satisfying
\ref{itm:cri1}, and $(\xi,\eta)$ satisfies \ref{itm:cri2},
then $f$ is an $m$-type edge.
Moreover, if $(\xi,\eta)$ also satisfies \ref{itm:cri3}--\ref{itm:cri4},  
then $f$ is an $(m,n)$-type edge.
\end{proposition}
As we will see, the conditions \ref{itm:cri2}--\ref{itm:cri4}
does not depend on the choice of 
null vector field $\eta$ satisfying \ref{itm:cri1}.
To show this fact, we show several lemmas which we shall need later.
Firstly we show that
the conditions does not depend on the choice of the diffeomorphism
on the target.
In what follows in this section, 
$f$ is as in Proposition \ref{prop:sufcond}.
\begin{lemma}\label{lem:targetnot}
Let\/ $\Phi$ be a diffeomorphism-germ on\/ $(\R^3,0)$,
and set\/ $\hat f=\Phi(f)$.
If\/ $f$ and\/ $(\xi,\eta)$ satisfy the condition\/ $C$, then\/
$\hat f$ and\/ $(\xi,\eta)$ satisfy\/ $C$, where\/ 
$C=\{\ref{itm:cri1}\}$,
$C=\{\ref{itm:cri1},\ref{itm:cri2}\}$,
$C=\{\ref{itm:cri1}$--$\ref{itm:cri3}\}$ and\/
$C=\{\ref{itm:cri1}$--$\ref{itm:cri4}\}$.
\end{lemma}
\begin{proof}
Let us assume $\eta$ satisfies \ref{itm:cri1}.
By a direct calculation, we have
$\eta\hat f=d\Phi(f)\eta f$, and
\begin{equation}\label{eq:etai}
\eta^i\hat f=\sum_{j=0}^{i-1}
c_{ij}\eta^j(d\Phi(f))\eta^{i-j} f\quad(c_{ij}\in\R\setminus\{0\}).
\end{equation}
By \ref{itm:cri1}, 
$\eta^i\hat f=0$ $(2\leq i\leq m-1)$ and
$\eta^m\hat f=d\Phi(f)\eta^m f$ on $S(f)$.
Then we see the assertion for the cases $C=\{\ref{itm:cri1}\}$
and $C=\{\ref{itm:cri1},\ref{itm:cri2}\}$.
We assume $\eta$ satisfies \ref{itm:cri1}--\ref{itm:cri3}.
By \eqref{eq:etai} and $c_{1n}=1 (\ne0)$ we see the assertion.
\end{proof}
It is clear that the conditions 
\ref{itm:cri2}--\ref{itm:cri4} do not depend 
on the choice of $\xi$, i.e., non-zero functional multiple and
extension other than $S(f)$. 
Moreover, it does not depend on the 
non-zero functional multiple of $\eta$:
\begin{lemma}\label{lem:etamulti}
Let $h$ be a non-zero function.
If $f$ and $(\xi,\eta)$ satisfy the condition $C$, then
$f$ and $(\xi,\hat\eta)$ satisfy $C$, where $\hat\eta=h\eta$ and
$C$ is the same as those in Lemma \ref{lem:targetnot}.
\end{lemma}
\begin{proof}
Since $(h\eta)^if$ is a linear combination of
$\eta f,\ldots,\eta^if$, and the coefficient of $\eta^if$ is $h^i$,
we see the assertion.
\end{proof}
A coordinate system $(u,v)$ satisfying
$S(f)=\{v=0\}$, $\eta|_{S(f)}=\partial_v$ is 
said to be {\it adapted\/}.
\begin{lemma}\label{lem:adaptedex}
Let $f:(\R^2,0)\to(\R^3,0)$ be a frontal-germ
satisfying $\rank df_0=1$.
We assume that the set of singular points $S(f)$ is a regular curve.
For any null vector field $\eta$, there exists an adapted coordinate system
$(u,v)$ such that $\eta=\partial_v$ for any $(u,v)$.
\end{lemma}
In this lemma, we do not assume that $f$ is an $m$-type edges.
\begin{proof}
Since $\rank df_0=1$,  one can easily see 
that there exists a coordinate system $(u,v)$ such
that $\eta=\partial_v$ for any $(u,v)$.
Since $S(f)$ is a regular curve, and the 
tangent direction of it is not in $\ker df_0$,
$S(f)$ can be parametrized as $(u,a(u))$.
Define a new coordinate system $(\tilde u,\tilde v)$ by
$\tilde u=u$ and $\tilde v=v-a(u)$.
Then $S(f)=\{\tilde v=0\}$ and $\partial/\partial \tilde v=\partial/\partial v$
hold.
This shows the assertion.
\end{proof}
\begin{lemma}\label{lem:etachange}
If two null vector fields $\eta,\tilde\eta$ satisfy \ref{itm:cri1},
and $(\xi,\eta)$ satisfies $C$, then 
$(\xi,\tilde\eta)$ also satisfies $C$.
Here, $C$ is the collection of the conditions:
$C=\{\ref{itm:cri2}\}$,
$C=\{\ref{itm:cri2},\ref{itm:cri3}\}$,
$C=\{\ref{itm:cri2}$--$\ref{itm:cri4}\}$.
\end{lemma}
\begin{proof}
Let us assume that $\eta$ and $\tilde\eta$ satisfy \ref{itm:cri1}.
Since the assumption \ref{itm:cri1} and 
the assertion do not depend on the choice of the
coordinate system on the source by Lemma \ref{lem:adaptedex}, 
we take $(u,v)$ an adapted coordinate system
with $\eta=\partial_v$ for any $(u,v)$
and $\xi=\partial_u$.
Since $f_v=\cdots=f_{v^{m-1}}=0$ on the $u$-axis,
$f_v$ has the form 
$f_v=v^{m-1}\psi(u,v)$.
If the pair $(\xi,\eta)$ satisfies \ref{itm:cri2}, then
$\rank(f_u,\psi)=2$ on the $u$-axis.
On the other hand, any null vector field is written as
$
a_1(u,v)\partial_u+a_2(u,v)\partial_v
$, $(a_1(u,0)=0,\,a_2(u,v)\ne0)$.
By Lemma \ref{lem:etamulti},
dividing this by $a_2$, we may assume an extended null vector
field $\tilde\eta$ is 
$$
\tilde\eta=va(u,v)\partial_u+\partial_v.
$$
Since it holds that $\tilde\eta^2f=0$ on the $u$-axis (when $m>2$) and
$f_u(u,0)\ne0$, we have $a(u,0)=0$. Continuing this argument,
we may assume
\begin{equation}\label{eq:tileta}
\tilde\eta=v^{m-1}a(u,v)\partial_u+\partial_v.
\end{equation}
Thus $\tilde\eta f=v^{m-1}(af_u+\psi)$ holds, and
$
\tilde\eta^mf=(m-1)!(af_u+\psi)
$ holds on the $u$-axis.
Therefore $(\xi,\tilde\eta)$ satisfies \ref{itm:cri2}.
We assume that the pair $(\xi,\eta=\partial_v)$ satisfies
\ref{itm:cri1}-\ref{itm:cri3} and $(\xi,\tilde\eta)$
does \ref{itm:cri1}, \ref{itm:cri2}.
By this assumption, 
$\rank(f_u,\psi)=2$,
$\rank(f_u,\psi,\psi_{v^i})=2$
$(0<i<n-m)$.
By the form of $\tilde\eta$, it holds that
$\tilde\eta^{m+1}f=(m-1)!(a_vf_u+a f_{uv}+\psi_v)$
on the $u$-axis.
Since $f_{v}=0$ on the $u$-axis, $f_{uv}=0$ on the $u$-axis.
Thus $\rank (\xi f,\tilde\eta^mf,\tilde\eta^{m+1}f)=2$ on the $u$-axis.
Similarly, $f_{v^{m-1}}=0$ on the $u$-axis, 
$f_{uv^2}=\cdots=f_{uv^{m-1}}=0$ on the $u$-axis.
Thus if $i\leq m-1$, then since $n<2m$, we have
$
\tilde\eta^{m+i}f=(m-1)!(\psi_{v^{i}}+a_{v^i}f_u)
$ on the $u$-axis.
Thereby we have
$\rank (\xi f,\tilde\eta^mf,\tilde\eta^{m+i}f)=2$ ($i\leq m-2$)
on the $u$-axis.
The last assertion can be shown by the same calculation.
\end{proof}
\begin{proof}[Proof of Proposition \ref{prop:sufcond}]
We assume $f$ satisfies the condition of the proposition, and $(\xi,\eta)$ 
satisfies
the conditions \ref{itm:cri1} and \ref{itm:cri2}.
Then we take an adapted coordinate system $(u,v)$ such that $\eta=\partial_v$.
By the proof of Lemma \ref{lem:etachange},
there exist $p(u)$ and $q(u,v)$ such that
$f(u,v)=p(u)+v^m q(u,v)$, and $(p_1)_u(0,0)\ne0$, where $p=(p_1,p_2,p_3)$.
We set $U=p_1(u),V=v$. Then $f$ has the form
$(U,P_2(U),P_3(U))+V^mQ(U,V)$.
By a coordinate change on the target,
$f$ has the form $(U,0,0)+V^mQ(U,V)$, 
where $Q(U,V)=(0,Q_2(U,V),Q_3(U,V))$.
Rewriting the notation, we may assume $f$ is written as
\[f(u,v)=(u,v^mq_2(u,v),v^mq_3(u,v)).\]
On this coordinate system, $\partial_v$ satisfies the condition
\ref{itm:cri1}, it satisfies \ref{itm:cri2} by Lemma \ref{lem:etachange}.
This implies $(q_2(0,0),q_3(0,0))\ne(0,0)$.
So we assume $q_2(0,0)\ne0$.
We set $U=u$, $V=vq_2(u,v)^{1/m}$.
Rewriting the notation, we may assume $f$ is written as
$(u,v^m,v^mq_3(u,v))$.
By a coordinate change on the target, we may assume $f$ is written as
$(u,v^m,v^{m+1}q_3(u,v))$.
This proves the first assertion.
We assume that $\eta$ also satisfies \ref{itm:cri3} and \ref{itm:cri4}.
We may assume $f$ is written as 
$(u,v^m,v^{m+1}q_3(u,v))$.
By Lemma \ref{lem:etachange}, we may assume that $\partial_v$ satisfies
\ref{itm:cri3} and \ref{itm:cri4}.
By \ref{itm:cri3}, the function $q_3(u,v)$ satisfies
$
q_3=(q_3)_v=\cdots(q_3)_{v^{n-m-1}}=0
$
on the $u$-axis.
Thus $f$ is written as 
$(u,v^m,v^{n}q_4(u,v))$.
By \ref{itm:cri4}, it holds that $q_4\ne0$,
and hence the assertion is proved.
\end{proof}
By the proof of Lemma \ref{lem:etachange}, we have the following property:
\begin{corollary}\label{cor:div}
Let $f:(\R^2,0)\to(\R^3,0)$ be a
frontal satisfying $\rank df_0=1$ and
the set of singular points $S(f)$ is a regular curve.
Furthermore, assume a vector field $\eta$ satisfying
the condition \ref{itm:cri1} with $\partial_v$ satisfying \ref{itm:cri1}.
Let $(u,v)$ be an adapted
coordinate system.
Then there exists $\psi$ such that
$
\eta f(u,v)=v^{m-1}\psi(u,v)
$.
\end{corollary}
\subsection{Normal form of $m$ or $(m,n)$-type edge}
Given a curve-germ $\gamma:(\R,0)\to(\R^2,0)$,
if there exists $m$ such that
$\gamma'=t^{m-1}\rho$ ($\rho(0)\ne0$),
then $\gamma$ at $0$ is said to be of {\it finite multiplicity\/},
and such an $m$ is called the {\it multiplicity\/} or the {\it order\/}
of $\gamma$ at $0$.
Moreover, if there exists $n$ 
($n>m$ and $n\ne km$, $k=2,3,\ldots$) such that
$\gamma$ is $\A^n$-equivalent to
$(t^m,t^n)$,
then $\gamma$ is called of 
$(m,n)$-{\it type\/}.
This $(m,n)$ is well-defined since
if $\gamma$ is $\A^r$-equivalent to $(t^m,0)$ then
it is not $\A^r$-equivalent to $(t^m,t^i)$ for $i\leq r$,
$i\ne km$ $(k=1,2,\ldots)$.
We simplify a curve-germ of $(m,n)$-type and
an $(m,n)$-type edge
by coordinate changes on the source and
by special orthonormal matrices on the target.
Let $(x,y)$ be the ordinary coordinate system of $(\R^2,0)$.
A coordinate system $(u,v)=(u(x,y),v(x,y))$ is {\it positive\/}
if the determinant of the Jacobi matrix of $(u(x,y),v(x,y))$ is positive.
We have the following results.
\begin{lemma}\label{lem:curvenormal}
Let $\gamma:(\R,0)\to(\R^2,0)$ be a curve germ satisfying
$\gamma^{(i)}(0)=0$ $(i=1,\ldots,m-1)$, and $\gamma^{(m)}(0)\ne0$.
Then there exist
a parameter $t$ and 
a special orthonormal matrix $A$ on $\R^2$ such that
\begin{equation}\label{eq:mnormal}
A \gamma(t)
=
\left(t^m,\ t^{m+1}b(t)\right).
\end{equation}
Let $\gamma:(\R,0)\to(\R^2,0)$ be a curve germ of $(m,n)$-type.
Then there exist
a parameter $t$ and 
a special orthonormal matrix $A$ on $\R^2$ such that
\begin{equation}\label{eq:mnnormal}
A \gamma(t)
=
\left(t^m,\ 
\sum_{i=2}^{\lfloor n/m\rfloor} a_it^{im}+
t^nb(t)
\right)\quad(b(0)\ne0),
\end{equation}
where $\lfloor k\rfloor$ is the greatest integer less than $k$ 
$($in our convention, $n/m$ is not an integer\/$)$.
\end{lemma}
\begin{proof}
One can easily see the first assertion.
We assume that $\gamma$ is a curve germ of $(m,n)$-type,
then we may assume $\gamma(t)=(t^m,t^{m+1}b(t))$.
If $t^{m+1}b(t)$ has a term $t^i$ $(i<n, i\ne km)$, then
$j^n\gamma(0)$ is not $\A^n$-equivalent to $(t^m,t^n)$.
This proves the assertion.
\end{proof}
\begin{proposition}\label{prop:edgenormal}
Let $f:(\R^2,0)\to(\R^3,0)$ be an $m$-type edge.
Then there exist
a positive coordinate system $(u,v)$ and
a special orthonormal matrix $A$ on $\R^3$ such that
\begin{equation}\label{eq:edgenormal1}
A f(u,v)
=
\left(u,
\dfrac{u^2a(u)}{2}+\dfrac{v^{m}}{m!},
\dfrac{u^2b_0(u)}{2}
+\dfrac{v^m}{m!}b_m(u,v)\right)\quad(b_m(0,0)=0).
\end{equation}
Moreover, if $f$ is an $(m,n)$-type edge,
then there exist
a positive coordinate system $(u,v)$ and
a special orthonormal matrix $A$ on $\R^3$ such that
\begin{equation}\label{eq:edgenormal2}
A f(u,v)
=
\left(u,
\dfrac{u^2a(u)}{2}+\dfrac{v^{m}}{m!},
\dfrac{u^2b_0(u)}{2}
+
\sum_{i=2}^{\lfloor n/m\rfloor}\dfrac{v^{im}}{(im)!}b_{im}(u)
+
\dfrac{v^{n}b_n(u,v)}{n!}\right),
\end{equation}
$b_{n}(0,0)\ne0$.
\end{proposition}
\begin{proof}
By the proof of Proposition  \ref{prop:sufcond},
we may assume 
\[f(u,v)=(u,u^2a_2(u)+v^{m}a_{2m}(u,v),
u^2a_3(u)+v^{m}a_{3m}(u,v)).\]
By that proof again, $(a_{2m}(0,0),a_{3m}(0,0))\ne(0,0)$.
By a rotation on $\R^3$, we may assume $a_{2m}(0,0)>0$ and $a_{3m}(0,0)=0$.
By a coordinate change
$v\mapsto va_{2m}(u,v)^{1/m}$, we may assume
$f(u,v)=(u,u^2a_2(u)+v^{m}/m!,u^2a_3(u)+v^{m}a_{3m}(u,v))$,
$(a_{3m}(0,0)=0)$.
This proves the first assertion.
If $f$ is an $(m,n)$-type edge,
then the function $a_{3m}(u,v)$ can be expanded by
$$
\sum_{i=0}^{n-1}v^ib_i(u)+v^{n}b_n(u,v).
$$
Since $f$ is an $(m,n)$-type edge,
the curve $v\mapsto f(u,v)$ is 
of $(m,n)$-type for any $u$ near $0$.
This implies that $b_i(u)=0$ 
$(i\ne km, k\geq1)$. By $a_{3m}(0,0)=0$,
$b_0(u)=0$.
This proves the assertion.
\end{proof}
Each form \eqref{eq:edgenormal1} and \eqref{eq:edgenormal2} is
called the {\it normal form\/} of an $m$-type edge and an $(m,n)$-type edge,
respectively.
Looking the first and the second components in 
\eqref{eq:edgenormal1} and \eqref{eq:edgenormal2}, 
we remark that 
the $m$-jet of the coordinate system $(u,v)$ which gives the normal form is 
uniquely determined up to $\pm$ when $m$ is even.
Let $f:(\R^2,0)\to(\R^3,0)$ be an $m$-type edge and $\eta$ a null vector field which
satisfies the condition \ref{itm:cri1}.
Then the subspace $V_1=df_0(T_0\R^2)$ and the subspace $V_2$ spanned by
$df_0(T_0\R^2)$, $\eta^mf(0)$ do not depend on the choice of $\eta$.
We assume that the representation $f=(f_1,f_2,f_3)$ of $xyz$-space $\R^3$
satisfies that
$V_1$ is the $x$-axis and $V_2$ is the $xy$-plane.
Then the coordinate system $(u,v)$ gives the normal form \eqref{eq:edgenormal1}
if and only if 
$f_1(u,v)=u$ and $(f_2)_{uv}$ is identically zero.
\subsection{Geometric invariants}
\subsubsection{Cuspidal curvatures}
Let $f$ be an $m$-type edge.
A pair of vector fields $(\xi,\eta)$ is said to be {\it adapted}
if $\xi$ is tangent to $S(f)$, and $\eta$ is a null vector field.
We take an adapted pair of vector fields $(\xi,\eta)$ such that
$\eta$ satisfies the condition \ref{itm:cri1}, and
$(\xi,\eta)$ is positively oriented. 
One can show the existence of such a pair by the definition of $m$-type edge.
We define
$$
\omega_{m,m+1}(t)
=
\frac{|\xi{f}|^{(m+1)/m}\det(\xi{f},\eta^{m}{f},\eta^{m+1}f)}
{|\xi{f}\times\eta^{m}{f}|^{(2m+1)/m}}(\mu(t))
$$
where $\mu$ is a parametrization of $S(f)$. 
We call $\omega_{m,m+1}$ the $(m,m+1)$-{\it cuspidal curvature}.
We have the following lemma:
\begin{proposition}\label{lem:notdepend}
The function $\omega_{m,m+1}$ does not depend on the choice of $(\xi,\eta)$ 
satisfying the condition \ref{itm:cri1}.
\end{proposition}
\begin{proof}
Since it is not appeared in the formula,
$\omega_{m,m+1}$ does not depend on the 
choice of the coordinate system.
Let $(\xi,\eta)$ be an adapted pair of vector fields
satisfying the condition \ref{itm:cri1}.
It is clear that 
the function $\omega_{m,m+1}$ does not depend on the choice
of $\xi$.
We take an adapted coordinate system $(u,v)$ satisfying
$\partial_v=\eta$.
Then 
$$\omega_{m,m+1}(u)=|f_u|^{(m+1)/m}\det(f_u,f_{v^m},f_{v^{m+1}})
|f_u\times f_{v^m}|^{-(2m+1)/m}.$$

By Corollary \ref{cor:div}, we have $f_v=v^{m-1}\psi$.
Let $\tilde\eta$ be another  null vector field satisfying
the condition \ref{itm:cri1}.
We see that
$\omega_{m,m+1}$ does not depend on 
the non-zero
functional multiples of $\eta$,
we may assume  
$\tilde\eta=a(u,v)\partial_u+\partial_v$.
By the proof of Lemma \ref{lem:etachange},
we may assume that $\tilde\eta$ is
\begin{equation}\label{eq:eta}
\tilde\eta=v^{m-1}a(u,v)\partial_u+\partial_v.
\end{equation}
Then by $f_v=v^{m-1}\psi$,
$$
\tilde\eta f=v^{m-1}(af_u+\psi).
$$
Thus 
\begin{equation}\label{eq:etam}
\tilde\eta^m f=(m-1)!(af_u+\psi)+(m-1)(m-1)!v\eta(af_u+\psi)+v^2g(u,v),
\end{equation}
where $g$ is a function, and
$$
\tilde\eta^{m+1} f=(m-1)!\eta(af_u+\psi)+(m-1)(m-1)!\eta v\eta(af_u+\psi)
=m!(\eta a f_u+a \eta f_u+\eta\psi)
$$
hold on the $u$-axis.
Since $\psi=((m-1)!)^{-1}f_{v^m}$ and 
$\psi_v=(m!)^{-1}f_{v^{m+1}}$, we have
\begin{align*}
\frac{|\xi{f}|^{(m+1)/m}\det(\xi{f},\eta^{m}{f},\eta^{m+1}f)}
{|\xi{f}\times\eta^{m}{f}|^{(2m+1)/m}((m-1)!)^{1/m}}(u,0)
=&
\frac{|f_u|^{(m+1)/m}\det(f_u,\psi,\psi_v)}
{|f_u\times \psi|^{(2m+1)/m}}(u,0)\\
=&
\frac{|f_u|^{(m+1)/m}\det(f_u,f_{v^m},f_{v^{m+1}})}
{|f_u\times f_{v^m}|^{(2m+1)/m}}(u,0).
\end{align*}
This shows the assertion.
\end{proof}
We have the following
proposition.
\begin{proposition}\label{prop:type:front}
Let $f\colon(\R^2,0)\to(\R^3,0)$ be an $m$-type edge. 
Then $f$ at $0$ is an $(m,m+1)$-type edge if and only if 
$\omega_{m,m+1}\neq0$ at $0$. 
\end{proposition}
\begin{proof}
Since $f$ is an $m$-type edge, by Proposition \ref{prop:edgenormal},
we may assume that $f$ is given by the right-hand side of 
\eqref{eq:edgenormal1}.
Since $b_m(0,0)=0$, there exist $c_1(u)$ and $c_2(u,v)$ such that
$b_m(u,v)=c_1(u)+vc_2(u,v)$.
Since we can take $\eta=\partial_v$, the function
$\omega_{m,m+1}$ is a non-zero functional multiple of
$c_2(u,0)$.
Then we see the assertion.
\end{proof}
It is easy to show that $(m,m+1)$-type edges are fronts and
that an $m$-type edge is a front if and only if
$\omega_{m,m+1}\ne0$.
In Appendix \ref{sec:genbias},
we define $(m,n)$-cuspidal curvature for a curve germ of $(m,n)$-type,
denoting it by $r_{m,n}$. 
An intersection curve of $(m,m+1)$-type edge $f$ with
$T$ as in Lemma \ref{lem:edgem},
is a curve-germ of $(m,m+1)$-type.
The following holds.
\begin{corollary}
Let $f\colon(\R^2,0)\to(\R^3,0)$ be a $\sigma$-edge, 
where $\sigma$ is $\mathcal{A}$-equivalent to $v\mapsto(v^m,v^{m+1})$. 
Then the $(m,m+1)$-cuspidal curvature $\omega_{m,m+1}$ at $0$ coincides with 
the $(m,m+1)$-cuspidal curvature $r_{m,m+1}$ of the intersection curve $\rho$ 
of $f$ with a plane $P$ which is perpendicular to the tangent line to $f$ at $0$.
\end{corollary}
\begin{proof}
By the assumption, we may assume that $f$ is given by the normal form
\eqref{eq:edgenormal1}. 
Since $f_u(0,0)=(1,0,0)$ and $f_v(0,0)=(0,0,0)$, the plane $P$ is given by $P=\{(0,y,z)\in\R^3\ |\ y,z\in\R\}$. 
Thus the intersection curve $\rho$ can be parametrized by 
$$\rho(v)=f(0,v)=\left(0,\frac{v^m}{m!},\frac{b_{m+1}(0,v)v^{m+1}}{(m+1)!}\right).$$
This can be considered as a normal form of a curve which is 
$\mathcal{A}^{m+1}$-equivalent to $v\to(v^m,v^{m+1})$. 
Hence we have the assertion by Example \ref{ex:curvefront}.
\end{proof}

Let $f$ be an $m$-type edge.
We assume $\omega_{m,m+1}$ is identically zero on $S(f)$.
Let $\mu(t)$ be a parametrization of $S(f)$.
We define
$$
\omega_{m,m+2}(t)=\dfrac{|\xi f|^{(m+2)/m}\det(\xi f,\eta^m f,\eta^{m+2}f)}
{|\xi f\times\eta^m f|^{(2m+2)/m}}(\mu(t)).
$$
We will see this does not depend on the choice of $(\xi,\eta)$ which
satisfies the conditions \ref{itm:cri1} and \ref{itm:cri2}
in Proposition \ref{prop:sufcond}
and
$\det(\xi{f},\eta^{m}f,\eta^{j}f)=0$ 
$(j<m+2)$.
Inductively, we define 
$\omega_{m,m+i}$ when $\omega_{m,m+j}=0$ $(j\leq i-1)$ by
$$
\omega_{m,m+i}=\dfrac{|\xi f|^{(m+i)/m}\det(\xi f,\eta^m f,\eta^{m+i}f)}
{|\xi f\times\eta^m f|^{(2m+i)/m}}(\mu(t)).
$$
We will also see this does not depend on the choice of $(\xi,\eta)$ 
satisfying the conditions \ref{itm:cri1} and \ref{itm:cri2}
in Proposition \ref{prop:sufcond}
and
$\det(\xi{f},\eta^{m}f,\eta^{j}f)=0$ 
$(j<m+i)$.
If $i=m$, 
we set $\beta_{m,2m}=\omega_{m,2m}$.
\begin{proposition}\label{lem:cuspialbias}
Under the assumption $\omega_{m,m+1}=\cdots=\omega_{m,m+i-1}=0$,
the function $\omega_{m,m+i}$ $(i=1,\ldots,m-1)$
does not depend on the choice of the pair $(\xi,\eta)$ which
satisfies the conditions \ref{itm:cri1} and \ref{itm:cri2}
in Proposition \ref{prop:sufcond}
and
$\det(\xi{f},\eta^{m}f,\eta^{m+j}f)=0$ 
$(1\leq j< i)$.
\end{proposition}
\begin{proof}
We already showed the case $i=1$ in Proposition \ref{lem:notdepend}. 
Let $(\xi,\eta)$ be a pair of vector fields satisfying
the assumption of lemma.
We take an adapted coordinate system $(u,v)$
such that $\partial_v=\eta$.
By the proof of Lemma \ref{lem:etachange},
we see
$f_v=v^{m-1}\psi$.

Moreover, we have:
\begin{lemma}\label{lem:psiv}
There exist functions $\alpha,\beta$,
and a vector valued function $\theta$ such that
\begin{equation}\label{eq:psiv}
\psi_{v}=\alpha f_u+\beta\psi+v^{i-1}\theta.
\end{equation}
\end{lemma}
\begin{proof}
Since $f_{v^{m+1}}=(m-1)!\psi_v$ on the $u$-axis,
$\omega_{m,m+1}=0$ implies that
there exists $\alpha_1,\beta_1,\theta_1$ such that
$\psi_{v}=\alpha_1f_u+\beta_1\psi+v\theta_1$.
We assume that there exist 
$\alpha_k,\beta_k,\theta_k$ such that
$\psi_{v}=\alpha_kf_u+\beta_k\psi+v^k\theta_k$
$(k=1,\ldots,i-2)$.
Differentiating this equation,
we have
$$
\psi_{v^{k+1}}
=
\sum_{l=0}^k{}\pmt{k\\l}
\Big((\alpha_k)_{v^l}f_{uv^{k-l}}
+
(\beta_k)_{v^l}\psi_{v^{k-l}}
+
(v^k)_{v^l}(\theta_k)_{v^{k-l}}\Big)
$$
Thus by $f_{uv}=\cdots=f_{uv^{m-1}}=0$ holds, and
$\psi_{v^{j}}\in \left\langle f_u,\psi\right\rangle_{\R}$ $(j\leq k)$
on the $u$-axis, we have
$
2=\rank(f_u,\psi,\psi_{v^{k+1}})=
\rank(f_u,\psi,\theta_k)$ on the $u$-axis.
Hence 
there exist functions $\alpha_{k+1},\beta_{k+1}$,
and a vector valued function $\theta_{k+1}$ such that
$\theta_k=\alpha_{k+1}f_u+\beta_{k+1}\psi+v\theta_{k+1}$.
This shows the assertion.
\end{proof}
We continue the proof of Proposition \ref{lem:cuspialbias}.
Since the assertion holds by multiplying
the null vector field by a non-zero function, we take 
a null vector field $\eta$ as in the right-hand side of \eqref{eq:eta}.
By the same calculations in the proof of 
Proposition \ref{lem:notdepend},
we have $\eta f=v^{m-1}(af_u+\psi)$. Thus
$$
\eta^{m+i}f
=
\sum_{k=0}^{m+i-1}
\pmt{m+i-1\\k}\eta^k v^{m-1}\eta^{m+i-1-k}(af_u+\psi).
$$
Since $\eta^k v^{m-1}=0$ if $k\ne m-1$ and $\eta^k v^{m-1}=(m-1)!$,
$$
\eta^{m+i}f
=
\pmt{m+i-1\\m-1}(m-1)!\eta^i(af_u+\psi).
$$
Thus
$
\eta^{m+i}f
=vg(u,v)+(af_u+\psi)_{v^i}$,
where $g$ is a function.
By $f_{uv}=\ldots=f_{uv^{m-1}}=0$ holds, and
$\psi_{v^{j}}\in \left\langle f_u,\psi\right\rangle_{\R}$ $(j\leq k)$
on the $u$-axis by Lemma \ref{lem:psiv}, we have
\begin{align*}
\dfrac{|\xi f|^{(m+i)/m}\det(\xi f,\eta^m f,\eta^{m+i}f)}
{|\xi f\times\eta^m f|^{(2m+i)/m}}(u,0)
=&
\dfrac{|f_u|^{(m+i)/m}\det(f_u,f_{v^m},f_{v^{m+i}})}
{|f_u\times f_{v^m}|^{(2m+i)/m}}(u,0),
\end{align*}
and this shows the assertion.
\end{proof}
We call $\omega_{m,m+i}$ the $(m,m+i)$-{\it cuspidal curvature}, and
$\beta_{m,2m}$ the $(m,2m)$-{\it bias}.
Note that $\beta_{m,2m}$ does not depend on the choice 
of $(\xi,\eta)$ satisfying \ref{itm:cri1}, \ref{itm:cri2} and 
$\inner{\xi{f}}{\eta^{m}{f}}=0$ at $p$ by the same calculation. 
In this case, $a(0,0)=0$ by the additional assumption.
If $f$ is an $m$-type edge,
and written as \eqref{eq:edgenormal1},
then 
$\omega_{m,m+1}(0)=(m+1)(b_m)_v(0,0)$.
If $f$ is an $(m,n)$-edge $(n<2m)$, 
and written as \eqref{eq:edgenormal2},
then 
$\omega_{m,n}(0)=b_n(0,0)$,
and $\beta_{m,2m}(0,0)=b_{2m}(0)$.
See Appendix \ref{sec:genbias} for geometric meanings of the terms $b_{im}$ 
$(i=2,\ldots,{\lfloor n/m\rfloor})$.
\subsubsection{Singular, normal curvatures and cuspidal torsion}
Let $f$ be an $m$-type edge, and $\mu(t)$ be a parametrization of
the singular set.
Let $\nu$ be a unit normal vector field of $f$, and we set
$\lambda=\det(f_u,f_v,\nu)$ for an oriented coordinate system $(u,v)$ on $(\R^2,0)$.
We set $\hat\mu=f\circ\mu$.
Then we define
\begin{equation}\label{eq:kskn}
\kappa_s(t)=\sgn\Big(\delta\ \eta^{m-1}\lambda(\mu(t))\Big)
\dfrac{\det(\hat\mu',\hat\mu'',\nu(\mu))}
{|\hat\mu'|^3},\quad
\kappa_\nu(t)=\dfrac{\inner{\hat\mu''}{\nu(\mu)}}
{|\hat\mu'|^2}
\end{equation}
and
\begin{equation}\label{eq:kt}
\kappa_t(t)
=
\dfrac{\det(\xi f,\eta^m f,\xi\eta^m f)}
{|\xi f\times\eta^m f|^2}(\mu(t))
-
\dfrac{\det(\xi f,\eta^m f,\xi^2 f)\inner{\xi f}{\eta^m f}}
{|\xi f|^2|\xi f\times\eta^m f|^2}(\mu(t)),
\end{equation}
where $\delta=1$ if $(\mu',\eta)$ agrees the orientation
of the coordinate system, and
$\delta=-1$ if $(\mu',\eta)$ does not agree the orientation.
We call $\kappa_s$, $\kappa_\nu$ and $\kappa_t$
{\it singular curvature, normal curvature} and {\it cuspidal torsion},
respectively.
These definitions are direct analogies of \cite{front,martins-saji}.
It is easy to see that the definitions \eqref{eq:kskn} do not
depend on the choice of parametrization of the singular curve.
Moreover, $\kappa_s$ does not depend on the choice of $\nu$,
nor the choice of $\eta$ when $m$ is even.
To see the well-definedness of $\kappa_t$, we need:
\begin{proposition}\label{lem:ktindep}
The definition \eqref{eq:kt} does not depend on the choice
of the adapted vector fields $(\xi,\eta)$ with $\eta$
which satisfies \ref{itm:cri1}.
\end{proposition}
\begin{proof}
One can easily to check it does not depend on the choice of
functional multiplications of $\eta$.
Since the assertion does not depend on the choice of
local coordinate system, one can choose an adapted coordinate system $(u,v)$
with $\partial_v$ satisfying \ref{itm:cri1}.
Let $\eta$ be a null vector field which satisfies \ref{itm:cri1}.
Then by the proof of Lemma \ref{lem:etachange},
we may assume $\eta$ is given by \eqref{eq:tileta}.
Then by \eqref{eq:etam}, we see $\eta^m f=(m-1)!(af_u+\psi)$ on the $u$-axis,
where $\psi$ is given in the proof of Lemma \ref{lem:etachange}.
Furthermore, by \eqref{eq:etam}, we see
and $\xi\eta^m f=(m-1)!(a_uf_u+af_{u^2}+\psi_u)$ on the $u$-axis.
Substituting these formulas into the right-hand side of \eqref{eq:kt},
we see it is
$$
\dfrac{\det(f_u,\psi,\psi_u)}
{|f_u\times \psi|^2}(u,0)
-
\dfrac{\det(f_u,\psi,f_{u^2})\inner{f_u}{\psi}}
{|f_u|^2|f_u\times \psi|^2}(u,0),
$$
and by $f_{v^m}=(m-1)!\psi$, this shows the assertion.
\end{proof}
If an $m$-type edge $f$ is given by the form 
\eqref{eq:edgenormal1}, then
$\kappa_s(0)=a(0)$,
$\kappa_\nu(0)=b(0)$ and
$\kappa_t(0)=(b_m)_u(0,0)$.
\subsection{Boundedness of Gaussian curvature and mean curvature near an $m$-type edge}
Here we study the behavior of the Gaussian and mean curvatures.

Let $g:(\R^i,0)\to\R$ be a function-germ ($i=1,2)$.
If there exists an integer $n$ $(n\geq1)$ 
such that $g\in \M_i^{n}$ and $g\not\in \M_i^{n+1}$,
then $g$ is said to be of {\it order\/} $n$, 
where $\M_i=\{g\colon(\R^i,0)\to\R\ |\ g(0)=0\}$ is the unique maximal ideal of the local ring 
of function-germs and $\M_i^{n}$ denotes the $n$th power of $\M_i$ 
(cf. \cite[p. 46]{ifrt-book}).
If $g\not\in \M_i$, then the order of $g$ is $0$.
The order of $g$ is denoted by $\ord(g)$.
If $g$ is of order $n$ $(n\geq 0)$, then $g$ is 
said to be of {\it finite order}.
Let $g_1,g_2:(\R^i,0)\to\R$ be two function-germs such that
$g_i$ is of finite order.
The {\it rational order\/} $\ord(f)$ of 
a function $f=g_1/g_2:(\R^i\setminus Z,0)\to\R$, where
$Z=g_2^{-1}(0)$ is 
$$
\ord(f)=\ord(g_1)-\ord(g_2).
$$
For a function $f=g_1/(|g_2|g_3):(\R^i\setminus Z,0)\to\R$, 
we define $\ord (f)=\ord(g_1)-\ord(g_2)-\ord(g_3)$,
where
$Z=g_2^{-1}(0)\cup g_3^{-1}(0)$.
If $g_1\in \M_i^\infty$, then we define $\ord(f)=\infty$.
If $\ord(f)=0$, then $f$ is called {\it rationally bounded}, and
$\ord(f)=1$, then $f$ is called {\it rationally continuous}
(\cite[Definition 3.4]{msuy}).
If $i=1$, this is the usual one.

Since the property 
$g\in \M_i^{n}$ does not depend on the choice of coordinate system,
the order and the rational order does not depend on the
choice of coordinate system.

Let $f\colon(\R^2,0)\to(\R^3,0)$ be an $m$-type edge, and 
let $(u,v)$ be an adapted coordinate system with $\partial_v$ satisfying
\ref{itm:cri1}.
We take $(m-1)!\psi$ in Corollary \ref{cor:div}. Namely, here we set $\psi$ by
$f_v=v^{m-1}\psi/(m-1)!$.
Since $f$ is an $m$-type edge, 
$f_u$ and $\psi$ is linearly
independent (Proposition \ref{prop:sufcond} and the independence of 
the condition \ref{itm:cri2}).
Thus the unit normal vector $\nu$ of $f$ can be taken as 
$\nu=\hat{\nu}/|\hat{\nu}|$ $(\hat{\nu}=f_u\times\psi)$.
Using $f_u$, $\psi$ and $\nu$, we define the following functions: 
\begin{align*}
\what{E}&=\inner{f_u}{f_u}, & \what{F}&=\inner{f_u}{\psi}, & \what{G}&=\inner{\psi}{\psi}, \\
\what{L}&=-\inner{f_u}{\what\nu_u}, & 
\what{M}&=-\inner{\psi}{\what\nu_u}, & \what{N}&=-\inner{\psi}{\what\nu_v}.
\end{align*} 
We note that coefficients of the first and the second fundamental forms
of $\sigma$-edges being of multiplicity $m$ can be written as 
\begin{align*}
E&=\what{E}, & F&=\dfrac{v^{m-1}}{(m-1)!}\what{F}, & G&=\left(\dfrac{v^{m-1}}{(m-1)!}\right)^2\what{G},\\
L&=\dfrac{\what{L}}{|\what\nu|}, & 
M&=\dfrac{v^{m-1}}{|\what\nu|(m-1)!}\what{M}, & 
N&=\dfrac{v^{m-1}}{(m-1)!|\what\nu|}\what{N}.
\end{align*}

\begin{lemma}
The differentials $\nu_u$ and $\nu_v$ of $\nu$ are written as 
\begin{align*}
\nu_u&=
-\dfrac{\what{G}\what{L}-\what{F}\what{M}}
{(\what{E}\what{G}-\what{F}^2)|\hat{\nu}|}f_u
-\dfrac{\what{E}\what{M}-\what{F}\what{L}}
{(\what{E}\what{G}-\what{F}^2)|\hat{\nu}|}\psi,\\
\nu_v&=
-\dfrac{\dfrac{v^{m-1}}{(m-1)!}\what{G}\what{M}-\what{F}\what{N}}
{(\what{E}\what{G}-\what{F}^2)|\hat{\nu}|}f_u
-\dfrac{\what{E}\what{N}-\dfrac{v^{m-1}}{(m-1)!}\what{F}\what{M}}
{(\what{E}\what{G}-\what{F}^2)|\hat{\nu}|}\psi.
\end{align*}
\end{lemma}
\begin{proof}
Since $\inner{\nu_u}{\nu}=\inner{\nu_v}{\nu}=0$, 
there exist functions $A,B,C,D$ on $(\R^2,0)$ such that 
$$\nu_u=Af_u+B\psi,\quad \nu_v=Cf_u+D\psi.$$
Considering $\inner{\nu_u}{f_u},\inner{\nu_u}{\psi},\inner{\nu_v}{f_u}$ and $\inner{\nu_v}{\psi}$, 
we have 
$$
-\frac{1}{|\hat{\nu}|}
\begin{pmatrix}
\what{L} \\ \what{M}
\end{pmatrix}=
\begin{pmatrix} 
\what{E} & \what{F} \\ \what{F} & \what{G} 
\end{pmatrix}
\begin{pmatrix} A \\ B \end{pmatrix},\quad 
-\frac{1}{|\hat{\nu}|}
\begin{pmatrix} \frac{v^{m-1}}{(m-1)!}\what{M} \\ \what{N} \end{pmatrix}=
\begin{pmatrix} 
\what{E} & \what{F} \\ \what{F} & \what{G} 
\end{pmatrix}
\begin{pmatrix} C \\ D \end{pmatrix}.
$$
Solving these equations, we have the assertion. 
\end{proof}

By this lemma, $\nu_v$ can be written as 
$$\nu_v=\frac{\what{N}}{(\what{E}\what{G}-\what{F}^2)|\hat{\nu}|}(\what{F}f_u-\what{E}\psi)$$
along the $u$-axis. 
Since $f_u$ and $\psi$ are linearly independent and $\what{E}\neq0$, 
the condition
$\nu_v(0)\neq0$ is equivalent to $\what{N}(0)\neq0$. 
To see this fact, we take
the same setting in the proof of Proposition \ref{lem:notdepend}.
Then we see
\begin{equation}\label{eq:hatn}
\det(f_u,f_{v^{m}},f_{v^{m+1}})
=m\det(f_u,\psi,\psi_v)=m\inner{\hat{\nu}}{\psi_v}=m\what{N}
\end{equation}
along the $u$-axis, 
where $\hat{\nu}=f_u\times\psi$ and 
$\what{N}=\inner{\hat{\nu}}{\psi_v}=-\inner{\hat{\nu}_v}{\psi}$. 
Since $\{f_u,\psi,\nu\}$ is a frame of $\R^3$, and
$\inner{f_u}{\nu_v}=\inner{f_v}{\nu_u}=0$,
$\inner{\nu}{\nu_v}=0$, it holds that
$\nu_v\ne0$ if and only if $\inner{\nu_v}{\psi}\ne0$.
Moreover, since $\inner{\nu}{\psi}=0$, it holds that
$\inner{\nu_v}{\psi}\ne0$ is equivalent to
$\inner{\hat{\nu}_v}{\psi}\ne0$.
Let $f:(\R^2,0)\to(\R^3,0)$ be an $(m,n)$-type edge,
and let us set
$$
r=\min\Big(\{n\}\cup
\{im\,|\,b_{im}(0)\ne0\text{ in the form }
\eqref{eq:edgenormal2}, i=2,3,\ldots\}\Big).
$$

\begin{theorem}\label{thm:ordkh}
Let $f:(\R^2,0)\to(\R^3,0)$ be an $(m,n)$-type edge.
Then the rational order of 
the mean curvature $H$ is $r-2m$.
If the normal curvature does not vanish at $0$,
then the rational order of 
the Gaussian curvature $K$ is $r-2m$.
\end{theorem}
\begin{proof}
We take an adapted coordinate system $(u,v)$ such that $\partial_v$
satisfies \ref{itm:cri1}.
Since $\hat L=\inner{f_{uu}}{\nu}$, the normal curvature does not vanish
if and only if $\hat L(0)\ne0$.
The Gaussian curvature $K$ and the mean curvature $H$ of $f$ are given by 
\begin{equation*}
K=\dfrac{(m-1)!}{v^{m-1}}\dfrac{\what{L}\what{N}-\frac{v^{m-1}}{(m-1)!}\what{M}^2}
{|\what\nu|^2(\what{E}\what{G}-\what{F}^2)},\quad
H=\dfrac{(m-1)!}{v^{m-1}}
\dfrac{\what{E}\what{N}-2\frac{v^{m-1}}{(m-1)!}\what{F}\what{M}+\frac{v^{m-1}}{(m-1)!}\what{G}\what{L}}{2|\what\nu|(\what{E}\what{G}-\what{F}^2)}.
\end{equation*}
This and $\what{E}\ne0$, $\what{E}\what{G}-\what{F}^2\ne0$ at $0$, 
together with
$$
\what N
=
\dfrac{v^{r-m-1}}{m(r-m-1)!}
(b_r(0)+v\alpha(u,v)),
$$
where $\alpha$ is a function,
by using the form \eqref{eq:edgenormal2} and \eqref{eq:hatn}
give the assertion.
\end{proof}

By Theorem \ref{thm:ordkh}, the order of $K$ and $H$
coincide. Moreover
since $n<2m$, they never bounded when the normal curvature
does not vanish.

\section{Curves passing through $m$-type edges}
In this section, we consider geometric invariants
of a curve $\gamma$ passing through an $m$-type edge $f$.
If $\hat\gamma=f\circ\gamma$ is non-singular,
then the usual invariants can be defined 
as well as the regular case.
We consider the case when $\hat\gamma$ has a singular point,
namely, $\gamma$ passing through a singular point of $f$ in the direction
of a null vector.

\subsection{Normalized curvatures of singular curves}
\label{sec:normcurv}
Following \cite{shibaume,fukui},
we introduce normalized curvature on curves in $\R^2$.
Let $\hat\gamma:(\R,0)\to(\R^n,0)$
be a curve, and let $0$ be a singular point.
We assume that there exists $k$ such that
$\hat\gamma'=t^{k-1}\rho$ ($\rho(0)\ne0$).

We set
\begin{equation}\label{eq:arclengt}
s=\int |\hat\gamma'|\,dt,
\end{equation}
and
\begin{equation}\label{eq:tilde:s}
\tilde s=\sgn(s)|s|^{1/k},
\end{equation}
we see $\tilde s$ is a $C^\infty$ function and
$d\tilde s/dt(0)>0$.
We call this parameter an $1/k$-{\it arc-length}.

\begin{proposition}
The parameter $t$ is an $1/k$-arc-length parameter of $\hat\gamma$ if and only if
	$|\hat\gamma'(t)|=k|t^{k-1}|$.
\end{proposition}
\begin{proof}
If $|\hat\gamma'(t)|=k|t^{k-1}|$ and $s(t)$ as in \eqref{eq:arclengt}
it holds that
$$
s(t)=\int_{0}^{t} k|\xi^{k-1}|d\xi=\int_{0}^{t}  \varepsilon \,  k \xi^{k-1}d\xi = \varepsilon \, t^k,\ 
\left(
\varepsilon=
\begin{dcases}
\sgn(t), \ & \text{if $k$ is even}\\ 
1 \ & \text{if $k$ is odd}
\end{dcases}
\right).
$$
Since $\sgn(s) = \sgn(t)$, then
$|s| = |t^k| $ and, therefore, $t= \sgn(s)|s|^{1/k}$.

Let us suppose now that $t$ is the $1/k$-arc-length, i.e., $t = \sgn(s) |s|^{1/k}$, with $s(t)$ as in \eqref{eq:arclengt}.
Since $\sgn (s) = \sgn (t)$, thus $t^k = \sgn(s)^k |s| =  \sgn(s)^{k+1}s$ and, consequently, $s'(t) = \sgn(t)^{k+1} k t^{k-1} = k |t|^{k-1}$. Therefore, it holds that $|\hat\gamma'(t)|=k|t^{k-1}|$.
\end{proof}

Let us set $n=2$.
Then the curvature $\kappa$ satisfies that
\begin{equation}
\tilde\kappa=|\tilde s^{k-1}|\kappa
\end{equation}
is a $C^\infty$ function.
We call $\tilde\kappa$ the {\it normalized curvature}.
This is originally introduced in \cite{shibaume} and
generalized in \cite{fukui}.
Let $f(t)$ be a given $C^\infty$-function, and $k\geq 2$ be an integer.
Then similarly to \cite[Theorem 1.1]{shibaume}, one can show that
there exists a unique plane curve up to isometries in $\R^2$ with 
normalized curvature given by 
$\tilde\kappa (t) = f(t)$, where $t$ is the $1/k$-arc-length parameter.

Using the frame $\{\e,\n\}$ along $\hat\gamma$ 
defined by $\e=\rho/|\rho|$ and $\n$ the $\pi/2$-rotation
of $\e$, the normalized curvature 
can be interpreted as follows:
We set the function $\kappa_1$ by
the equation
\begin{equation}\label{eq:frame:k1}
\pmt{\e'\\ \n'}=\pmt{0&\kappa_1\\ -\kappa_1&0}
\pmt{\e\\ \n},
\end{equation}
where ${}'$ is a differentiation by the $1/k$-arc-length.
Then we have:
\begin{proposition}
Let $\{\e,\n\}$ be the above frame along $\hat\gamma (t)$ in the Euclidean plane 
$\R^2$ satisfying \eqref{eq:frame:k1}, where $t$ is the $1/k$-arc-length parameter.
Then $\kappa_1=k \tilde\kappa$ holds.
\end{proposition}
\begin{proof} 
Since    $\hat\gamma' (t) = t^{k-1} \rho(t)$ where $\rho(0)	\neq  0$ and the $1/k$-arc-length parameter $t$ satisfies $|\hat\gamma'(t)|= k|t|^{k-1}$, so $
\hat\gamma''(t)=(k-1) t^{k-2}\rho(t)+t^{k-1}\rho'(t)
$ and $|\rho(t) |=k$. Then
\begin{equation*}%\label{eq:k:mppca}
\kappa(t)=\dfrac{1}{k^3|t|^{k-1}} \det(\rho(t),\rho'(t)).
\end{equation*}
Consequently, 
\begin{equation*}%\label{eq:tildek}
\tilde{\kappa}(t) = |t|^{k-1} \kappa(t) = \dfrac{1}{k^3} \det(\rho(t),\rho'(t)).
\end{equation*}
On the other hand, since $\kappa_1 (t) = \e'(t) \cdot \n(t)$, 
where
$\e(t)=\rho(t)/|\rho(t)| = \rho(t)/k$ and $\n(t)$ is 
the $\pi/2$-counterclockwise rotation of $\e(t)$, and the dot `$\cdot$' denotes the canonical inner product of $\R^2$, it holds that:	
\begin{equation*}%\label{eq:k1}
\kappa_1(t) = \frac{1}{k} \rho'(t) \cdot \n(t) =  \frac{1}{k^2}\det(\rho(t), \rho'(t)).
\end{equation*}	
Thus we have the assertion.
\end{proof}

\subsection{Normalized curvatures on frontals}\label{sec:curvatures}
According to Section \ref{sec:normcurv},
we define the normalized curvatures for curves on a frontal.
Let $f:(\R^2,0)\to(\R^3,0)$ be a frontal and 
$\nu$ a unit normal vector field of $f$.
Let 
$\gamma\colon (\R,0)\to (\R^2,0)$ be a curve.
We set $\hat\gamma=f\circ\gamma$.
We assume there exists $k$ such that
$\hat\gamma'=t^{k-1}\rho$ ($\rho(0)\ne0$).
%Then by the same discussion as in 
%Section \ref{sec:normcurv}, we set
%\begin{equation}
%s=\int |\hat\gamma'|\,dt,
%\end{equation}
%and
%\begin{equation}\label{eq:tilde:s}
%\tilde s=\sgn(s)|s|^{1/k}.
%\end{equation}
%Then we see $\tilde s$ is a $C^\infty$ function and
%$d\tilde s/dt(0)>0$.
The geodesic curvature $\kappa_g$, the normal curvature 
$\kappa_n$ and the
geodesic torsion $\tau_g$ are defined by
\begin{equation}
\kappa_g=\dfrac{\det(\hat\gamma',\hat\gamma'',\nu)}{|\hat\gamma'|^3},\ 
\kappa_n=\dfrac{\inner{\hat\gamma''}{\nu}}{|\hat\gamma'|^2}	,\ 
\tau_g=\dfrac{\det(\hat\gamma',\nu,\nu')}{|\hat\gamma'|^2}
\end{equation}
on regular points (see \cite[page 261]{docarmo}).
These curvatures can be unbounded near  singular points. 
Indeed, it holds that 
\begin{equation}\label{eq:reg:normaliz:curv}
\kappa_g=\dfrac{1}{|t|^{k-1}}\dfrac{\det(\rho,\rho',\nu)}{|\rho|^3}, \ 
\kappa_n=\dfrac{1}{t^{k-1}}\dfrac{\inner{\rho'}{\nu}}{|\rho|^2},\ 
\tau_g= \dfrac{1}{t^{k-1}} \dfrac{\det(\rho,\nu,\nu')}{|\rho|^2}.
\end{equation} 
One can easily see that 
\begin{equation}\label{eq:sing:normalz:curv}
\tilde\kappa_g=|\tilde s|^{k-1}\, \kappa_g,\ 
\tilde\kappa_n=\tilde s^{k-1}\, \kappa_n,\ 
\tilde\tau_g=\tilde s^{k-1}\, \tau_g
\end{equation}
are $C^\infty$ functions, where $\tilde{s}$ is the function given by \eqref{eq:tilde:s} for $\hat{\gamma}$. We call $\tilde\kappa_g, \tilde\kappa_n,  \tilde\tau_g$   {\it normalized} geodesic curvature, normal curvature and  geodesic torsion of $\tilde{\gamma}$, respectively.
These satisfy:
%%%%%%%%%%%%%%%%%%%%%%%%%%%%%%%%%%%%%%%%%%%%%%%%%%%%%%%%%%%%%%
%%%%%%%%%%%%%%%%%%%%%%%%%%%%%%%%%%%%%%%%%%%%%%%%%%%%%%%%%%%%%%
\begin{lemma}
It holds that 
\begin{align}
\tilde\kappa_g&=
\dfrac{1}{k^2\, k!^{-1/k}}
\dfrac{\det\left(\hat\gamma^{(k)},\hat\gamma^{(k+1)},\nu\right)}
{|\hat\gamma^{(k)}|^{2+1/k}},\\[2mm]
\tilde\kappa_n&=
\dfrac{1}
{k^2\, k!^{-1/k}}
\dfrac{\inner{\hat\gamma^{(k+1)}}{\nu}}
{|\hat\gamma^{(k)}|^{1+1/k}},\\[2mm]
\tilde\tau_g&=\dfrac{1}{k\, k!^{-1/k}}
\dfrac{\det(\hat\gamma^{(k)},\nu,\nu')}
{|\hat\gamma^{(k)}|^{1+1/k}}
\end{align}
at $t=0$. 
\end{lemma}
%%%%%%%%%%%%%%%%%%%%%%%%%%%%%%%%%%%%%%%%%%%%%%%%%%%%%%%%%%%%%%
%%%%%%%%%%%%%%%%%%%%%%%%%%%%%%%%%%%%%%%%%%%%%%%%%%%%%%%%%%%%%%
\begin{proof} 
Since $\hat\gamma'(t)=t^{k-1}\rho(t)$, then $\rho(0) = \frac{\hat\gamma^{(k)}(0)}{(k-1)!}$, $\rho'(0)=\frac{\hat\gamma^{(k+1)}(0)}{k!}$ and  $\rho_0=|\rho(0)|=\frac{|\hat\gamma^{(k)}(0)|}{(k-1)!}$. Therefore, it holds that
\begin{equation*}
\tilde s^{k-1}=t^{k-1} \left(\dfrac{\rho_0^{(k-1)/k}}{k^{(k-1)/k}} +  tO(t) \right)
\end{equation*}
where $O(t)$ is a smooth function of $t$.
Thus by \eqref{eq:reg:normaliz:curv} and \eqref{eq:sing:normalz:curv}, we get at $t=0$:
\[
\begin{aligned}
\tilde{\kappa}_g&=
\frac{\rho_0^{\frac{k-1}{k}} k^{1/k}}{k}\dfrac{\det(\rho,\rho',\nu)}{\rho_0^3}
=\dfrac{k^{1/k}}{k} \, \dfrac{\det(\rho,\rho',\nu)}{\rho_0^{2+1/k}}\\
&=\dfrac{k^{1/k}(k-1)!^{2+1/k}}{k\, (k-1)!\,k!}\dfrac{\det(\hat\gamma^{(k)},\hat\gamma^{(k+1)},\nu)}{ |\hat\gamma^{(k)}|^{2+1/k}}
=\dfrac{k!^{1/k}}{k^2}\dfrac{\det(\hat\gamma^{(k)},\hat\gamma^{(k+1)},\nu)}{|\hat\gamma^{(k)}|^{2+1/k}},
\end{aligned}
\]
\[
\begin{aligned}
\tilde{\kappa}_n&=\frac{\rho_0^{(k-1)/k}}{k^{(k-1)/k}} \,  \frac{\inner{\rho'}{\nu}}{\rho_0^2}
	=\frac{k^{1/k}}{k}\frac{\inner{\rho'}{\nu}}{\rho_0^{1+1/k}}\\
	&=\frac{k^{1/k}(k-1)!^{1+1/k}}{k\,k!}\frac{\inner{\hat\gamma^{(k+1)}}{\nu}}{|\hat\gamma^{(k)}|^{1+1/k}}
	=\frac{k!^{1/k}}{k^2}\frac{\inner{\hat\gamma^{(k+1)}}{\nu}}{|\hat\gamma^{(k)}|^{1+1/k}}
\end{aligned}
\]
and
\[
\begin{aligned}
\tilde{\tau}_g&=\frac{\rho_0^{(k-1)/k}}{k^{(k-1)/k}}\dfrac{\det(\rho,\nu,\nu')}{\rho_0^2}=\dfrac{k^{1/k}}{k}\dfrac{\det(\rho,\nu,\nu')}{\rho_0^{1+1/k}}\\&
=\dfrac{k^{1/k}(k-1)!^{1+1/k}}{k(k-1)!}\dfrac{\det(\hat\gamma^{(k)},\nu,\nu')}{|\hat\gamma^{(k)}|^{1+1/k}}
=\dfrac{k!^{1/k}}{k}\dfrac{\det(\hat\gamma^{(k)},\nu,\nu')}{|\hat\gamma^{(k)}|^{1+1/k}},
\end{aligned}
\]
which show the assertion.
\end{proof}

Similar with the case of plane curves, these invariants 
can be interpreted as follows.
Under the same assumption above,
we set
$e=\rho/|\rho|$,
$\nu=\nu(\hat\gamma)$ and
$b=-e\times\nu$.
Then $\{e,\nu,b\}$ is a frame along $\hat\gamma$.
We define $\kappa_1,\kappa_2,\kappa_3$ by
\begin{equation}
\pmt{
e'\\ b' \\ \nu'}
=
\pmt{
0&\kappa_1&\kappa_2\\
-\kappa_1&0&\kappa_3\\
-\kappa_2&-\kappa_3&0}
\pmt{
e\\ b \\ \nu},
\end{equation}
where $'=d/dt$ 
is the differentiation by the $1/k$-arc-length parameter.
With the above notation, we get the following:
%%%%%%%%%%%%%%%%%%%%%%%%%%%%%%%%%%%%%%%%%%%%%%%%%%%%%%%%%%%%%%
\begin{proposition}
If $t$ is the $1/k$-arc-length parameter, then 
	$$\kappa_1= k  \tilde\kappa_g, \  \kappa_2= k  \tilde\kappa_n \ \
\text{and}\ \  \kappa_3 = k  \tilde{\tau}_g $$ holds,
for any $t$.
\end{proposition}
%%%%%%%%%%%%%%%%%%%%%%%%%%%%%%%%%%%%%%%%%%%%%%%%%%%%%%%%%%%%%%
\begin{proof}
The $1/k$-arc-length parameter $t$ satisfies  $|\hat\gamma'(t)|= k|t^{k-1}|$. Then  $|\rho(t)| = k$ and $e=\rho / k$. 
So, putting $\tilde{s} = t$ at \eqref{eq:sing:normalz:curv} and using \eqref{eq:reg:normaliz:curv} , it holds that
\[ \begin{aligned}
 \kappa_1&= \inner{e'}{b} =\dfrac{1}{k^2}\det(\rho,\rho',\nu) = k \tilde{\kappa}_g, \\[2mm]
\kappa_2&= \inner{e'}{\nu} =\dfrac{1}{k}\inner{\rho'}{\nu} = k \tilde{\kappa}_n \,,\\[2mm]
\kappa_3&= -\inner{\nu'}{b}  =\dfrac{1}{k}\det(\rho,\nu,\nu')= k \tilde{\tau}_g.
\end{aligned}
\]
Thus the assertion holds.
\end{proof}
\subsection{Behaviors of $\kappa_g,\kappa_n$ and $\tau_g$ passing through an $m$-type edge}

In this section we shall study the orders of the 
geodesic and normal curvatures and the geodesic torsion of 
a curve passing through an $m$-type edge, concluding on boundedness.
Describing the condition, we use the curvature of such curve.
Let  $f:(\R^2,0) \to (\R^3, 0)$ be an $m$-type edge, 
$m\geq2$, and $\gamma:(\R,0)\to (\R^2, 0)$  
be a regular curve such that $\gamma'(0)$ is a null vector of $f$ at $0$.
Let $(u,v)$ be a coordinate system which gives the form 
\eqref{eq:edgenormal1}, and $\tilde\gamma(t)=(u(t),v(t))$ be
a parametrization of $\gamma$ where the coordinate system on
the target space is $(u,v)$, and the orientation of $\tilde\gamma$ agrees
the direction of $v$ at $0$.
Since such coordinate system is unique 
(unique up to $(u,v)\mapsto(u,-v)$ if $m$ is even), 
the order of contact of $\tilde\gamma$ with the $v$-axis at $0$
and
the curvature $\tilde\kappa$ of $\tilde\gamma$ is well-defined as a curve
on $f$.
We call such order of contact
the {\it order of contact with the normalized null direction},
and
we call $\tilde\kappa$ the {\it curvature written in the normal form}.
If $\tilde\gamma(t)=(t^lc(t),t)$ ($c(0)\ne0$),
then 
the order of contact with the normalized null direction is $l$,
and
$\tilde\kappa^{(l-2)}(0) = -l!c(0)$ and $\tilde\kappa^{(l-1)}(0) = -(l+1)!c'(0)$
hold.

\begin{theorem}\label{thm:ord1}
Let  $f:(\R^2,0) \to (\R^3, 0)$ be an $m$-type edge, 
$m\geq2$, and $\gamma:(\R,0)\to (\R^2, 0)$  
be a regular curve with order of contact $l \geq 2$ with 
the null direction of $f$ at $0$ and $\tilde\kappa$ 
the curvature of $\gamma$ written in the normal form of $f$.
Then, it holds that:

{\rm (1)}  The case $l\geq m$. \\ For $\kappa_g$, 
	\begin{itemize}
	\item if $m<l\leq2m$, then
	$\ord\kappa_g= l-2m$;
	\item if $l>2m$, then $\ord\kappa_g\geq 1$,
	and $\ord\kappa_g=1$ is equivalent to
	\[
	\begin{dcases}
	(l-1)!\kappa_t(0) \omega_{m,m+1}(0)-m!(m+1)!\tilde\kappa^{(l-2)}(0)\ne0 & (\text{if  $l=2m+1$}),\\ 
	\kappa_t(0)\omega_{m,m+1}(0)\ne0 & (\text{if $l>2m+1$});
	\end{dcases}
	\]
	\item if $l=m$,
	then $\ord\kappa_g\geq1-m$, and
	$\ord\kappa_g=1-m$ if and only if
	$\tilde\kappa^{(l-1)}(0)\ne0$.
	\end{itemize}
	For $\kappa_n$,
	it holds that $\ord\kappa_n\geq1-m$, and
	$\ord\kappa_n=1-m$ if and only if
	$\omega_{m,m+1}(0)\ne0$.
	For $\tau_g$, 
	
	\begin{itemize}
	\item if $l< 2m$, then 
	$\ord\tau_g\geq l-2m+1$, and
	$\ord\tau_g=l-2m+1$ is equivalent to 
	\[\begin{dcases}
	\omega_{m,m+1}(0)\ne0 & (\text{if $l< 2m-1$}),\\
	m(l-1)!\kappa_t(0)+{(m-1)!^2}\,\tilde\kappa^{(l-2)}(0)\, \omega_{m,m+1}(0)\ne0& 
	(\text{if $l= 2m-1$});
	\end{dcases}
	\]
	\item if $l\geq2m$, then 
	$\ord\tau_g\geq 0$, and
	$\ord\tau_g=0$ if and only if $\kappa_t(0)\ne0$.
	\end{itemize}
	
	{\rm (2)} The case $m/2<l<m$.\\
	For this case, it holds that $\ord\kappa_g=m-2l$, $\ord\kappa_n\geq m-2l+1$, and
	$\ord\kappa_n=m-2l+1$ 
	is equivalent to 
	\[\begin{dcases}
	\omega_{m,m+1}(0)\ne0 & (\text{if $l>(m+1)/2$}),\\
	(m+1)!(m-l+1)\kappa_\nu(0) (\tilde\kappa^{(l-2)})^2(0)+2l!^2 \omega_{m,m+1}(0)\ne0 & (\text{if $l=(m+1)/2$}).
	\end{dcases}\]
	For $\tau_g$,
	it holds that $\ord\tau_g\geq1-l$, and
	$\ord\tau_g=1-l$ 
	is equivalent to $\omega_{m,m+1}(0)\ne0$.
	
	{\rm (3)} The case $l\leq m/2$.\\
	In this case, it holds that $\ord\kappa_g\geq0$, and
	$\ord\kappa_g=0$ is equivalent to 
	
	\[\begin{dcases}
	\kappa_s(0)\ne0 & (\text{if $l<m/2$}),\\
	m!\kappa_s(0)(\tilde\kappa^{(l-2)})^2(0)+2l!^2\ne0 & (\text{if $l=m/2$}).
	\end{dcases}\]
	For $\kappa_n$ it holds that $\ord \kappa_n\geq0$, and
	$\ord\kappa_n=0$ if and only if $\kappa_n(0)\ne0$.
	For $\tau_g$,
	it holds that $\ord\tau_g\geq1-l$, and
	$\ord\tau_g=1-l$ if and only if $\omega_{m,m+1}(0)\ne0$.
\end{theorem}
{If $m$ is even and $(u,v)$ be a coordinate system 
which gives the form \eqref{eq:edgenormal1}, then $(u,-v)$ also
gives \eqref{eq:edgenormal1}.
In this case, changing $(u,v)$ to $(u,-v)$, the signs of 
$\tilde\kappa$ and $\omega_{m,m+1}$ reverse, 
and them of $\kappa_t$ and $\kappa_s$ do not change.
So, when $m$ is even, neither the condition
\[
\begin{aligned}
(l-1)!\kappa_t(0) \omega_{m,m+1}(0)-m!(m+1)!\tilde\kappa^{(l-2)}(0)&\ne0,\\
m(l-1)!\kappa_t(0)+(m+1)!^2\,\tilde\kappa^{(l-2)}(0)\, \omega_{m,m+1}(0)&\ne0
\quad\text{nor}\\
m!\kappa_s(0)(\tilde\kappa^{(l-2)})^2(0)+2l!^2&\ne0
\end{aligned}
\]
changes under the coordinate change $(u,v)$ to $(u,-v)$.
\begin{proof}	
Let $\hat\gamma=f\circ\gamma$. One can
assume that $f$ is given by the form \eqref{eq:edgenormal1}
	and, since $\partial v$ is a null vector of $f$, 
one  can take  $\gamma(t)=(x(t),t)$, with  $x(0)=x'(0)=0$.
Then $x(t)$ is of  order $l$ and we set $\gamma(t)=(t^lc(t),t)$ ($c(0)\ne0$).

 In the normal form 
(2.5), since $b_m(0,0)=0$, further we may 
assume $f$ is given by
$f(u,v)=
(u,u^2a(u)/2+v^{m}/m!,
u^2b_0(u)/2
+(v^m/m!)(ub_{m1}(u)+vb_{m2}(u,v)))
$. 
We recall that $\kappa_s(0)= a(0), \kappa_\nu(0) = b(0)$ and 
$\kappa_t(0)=(b_m)_u(0,0)$. 
Furthermore, it holds that $\kappa_t(0)=b_{m1}(0)$, 
$\omega_{m,m+1}= (m+1)b_{m2}(0,0)$,
$\tilde\kappa^{(l-2)}(0) = -l!c(0)$ and $\tilde\kappa^{(l-1)}(0) = -(l+1)!c'(0)$.
	We set $\phi$ by $f_v=v^{m-1}\phi/(m-1)!$.
	Then $\tilde\nu_2=f_u\times \phi$ gives a non-zero normal vector field to $f$.
	\medskip

\noindent (1).
Assume $l \geq m$.
By \eqref{eq:edgenormal1} we get $\hat\gamma = t^m \tilde\rho$, where
\[
\begin{aligned}
\tilde\rho(t) &=(t^{l - m} c(t), g_2(t), tg_3(t)),\\
g_2(t) &=
\dfrac{m!t^{2 l-m} a(t) c(t)^2+2}{2 m!},\\
g_3(t) &=
\dfrac{m!t^{2 l-m-1} b_0(t) c(t)^2
+2b_{m2}(t)+2t^{l-1} b_{m1}(t) c(t)}{2 m!}.
\end{aligned}
\]
Then $\hat\gamma = t^{m-1} \rho$, where $\rho = m\tilde\rho+t\tilde\rho'$. Note that $\rho (0) \neq 0$.
Setting $\nu_2(t)=\tilde\nu_2(\gamma(t))$, we can show that 
 $\nu_2(t)=(t^m d(t),te(t),1)$, where
\small
\begin{align}
\label{eq:d0100}
d(t)&=\dfrac{1}{2 m m!}
 \bigg(-2 m  b_{m1}  + 
   2 m t^{2 l - m}
     a b_{m1} c^2  m! 
- 2 m t^{l - m} b_0 c  m! \\
   &\hspace{20mm}+ 
   2 t^{1 + l - m}
     a b_{m2} c  m! + 
   2 m t^{1 + l - m}
     a b_{m2} c  m! + 
   m t^{3 l - m}
     b_{m1} c^3 m! a' \nonumber\\
   &\hspace{20mm}+
    t^{1 + 2 l - m}
     b_{m2} c^2 m! a' + m t^{1 + 2 l - m}
     b_{m2} c^2  m! a' - m t^{2 l - m}
     c^2 m! b_0' \nonumber\\
   &\hspace{20mm}+ 
   2 t^{2 + l - m} a c  m! 
b_{m2,v} 
+ 
   t^{2 + 2 l - m} c^2 m! a' 
b_{m2,v} - 2 m t^l c 
b'_{m1} - 2 m t 
b_{m2,u}\bigg),\nonumber\\
e(t)&=\dfrac{-1}{m}
\Big(m t^{l-1} b_{m1} c + (1+m)b_{m2}  + 
  t
b_{m2,v}\Big).\nonumber
\end{align}
\normalsize
We abbreviate the variable, namely $a=a(t)$, $b_{m2}=b_{m2}(\gamma(t))$ 
for instance, and $(b_{m2})_v=b_{m2,v}$.
Here, we see
\[
\begin{aligned}
&g_2(0)=\dfrac{1}{m!},\ g_2'(0)=0,\ 
g_3(0)=\dfrac{b_{m2}(0)}{m!},\ d(0)=\dfrac{-b_{m1}(0)}{m!}\ (\text{if}\ m<l),\\
&
d(0)=\dfrac{-m!b_0(0)c(0)-b_{m1}(0)}{m!}\ (\text{if}\ m=l),\ 
e(0)=-\dfrac{(m+1)b_{m2}(0)}{m}.
\end{aligned}
\]
To see the rational order of the invariants $\kappa_g, \kappa_n$, $\tau_g$ at $0$, 
we may use $\nu_2(t)$ instead of $\nu\circ\gamma(t)$
in \eqref{eq:reg:normaliz:curv}.
Since $g_2'(0)=0$, we can write $g_2'= t\tilde{g}$. We see
\begin{align}
\rho&=(lt^{l-m}c+t^{l-m+1}c',mg_2+t^2\tilde g_2,(m+1)tg_3+t^2g_3'),\label{eq:rho0200}\\
\rho'&=
\begin{dcases}
\Big(l(l-m)t^{l-m-1}c+t^{l-m}O(1),\\
\hspace{20mm} tO(1),(m+1)g_3+tO(1)\Big)\quad (l>m)\\
\Big((m+1) c'+tO(1),\\
\hspace{20mm} tO(1),(m+1)g_3+tO(1)\Big)\ (m=l),
\end{dcases}\label{eq:rho0300}
\end{align}
where $O(1)$ means a smooth function depending on $t$.
Then we see $|\rho,\rho',\nu_2|$, where $|\cdot|=\det(\cdot)$ is, for $l >m$: 
\begin{equation}\label{eq:kgt0100}
\vmt{
lt^{l-m}c+t^{l-m+1}O(1)
&
l(l-m)t^{l-m-1}c+t^{l-m}O(1)
&
t^md\\
mg_2+tO(1)&tO(1)&te\\
(m+1)tg_3+t^2O(1)&(m+1)g_3+tO(1)&1}.
\end{equation}
If $2m-l+1>0$, then \eqref{eq:kgt0100} is
$t^{l-m-1}A_1(t)$, where
\[
A_1(0)=
\vmt{
0&l(l-m)c(0)&0\\
mg_2(0)&0&0\\
0&(m+1)g_3(0)&1}=-\dfrac{l(l-m)}{(m-1)!}c(0).
\]
 If $2m-l+1=0$, then  \eqref{eq:kgt0100} is
$t^{m}A_2(t)$, where 
\[
\begin{aligned}
A_2(0)&=
\vmt{
0&l(m+1)c(0)&d(0)\\
mg_2(0)&0&0\\
0&(m+1)g_3(0)&1}=
-m (m+1) g_2(0)(lc-d g_3)(0)\\
&=
-\dfrac{m+1}{(m-1)!}
\left(
\dfrac{b_{m1} b_{m2}}{(m!)^2}+lc
\right)(0).
\end{aligned}
\]
If $2m-l+1<0$, then $l>m$ and \eqref{eq:kgt0100} is $t^mA_3(t)$, where
\[
\begin{aligned}
A_3(0)&=
\vmt{
0&0&d(0)\\
mg_2(0)&0&0\\
0&(m+1)g_3(0)&1}=m(m+1)d(0)g_2(0)g_3(0)\\
&=
-\dfrac{m (1+m) }{(m!)^3}b_{m1}(0) b_{m2}(0).
\end{aligned}
\]
If $m=l$, by \eqref{eq:rho0200} and \eqref{eq:rho0300},
we see the assertion, once $\ord|t|^{m-1}=m-1$ and $\ord |\rho|^3 = 0$.
These shows the assertion for $\kappa_g$.

Since one can easily see that $\inner{\rho'}{\nu_2}=(m+1)b_{m2}/m!$ at $0$,
the assertion for $\kappa_n$ is proved.
Next we see $|\rho,\nu_2,\nu_2'|$ is
\begin{equation}\label{eq:tgt0100}
\vmt{
lt^{l-m}c+t^{l-m+1}O(1)
&
t^md
&
mt^{m-1}d+t^mO(1)\\
mg_2+tO(1)&te&e+tO(1)\\
(m+1)tg_3+t^2O(1)&1&0}.
\end{equation}
If $2m-l-1>0$, then \eqref{eq:tgt0100} is
$t^{l-m}B_1(t)$, where
\[
B_1(0)=
\vmt{
lc(0)&0&0\\
mg_2(0)&0&e(0)\\
0&1&0}=-lc(0)e(0)=\dfrac{l(m+1)}{m}b_{m2}(0)c(0).
\]
If $2m-l-1=0$, then $m<l$ and \eqref{eq:tgt0100} is
$t^{m-1}B_2(t)$, where 
\[
\begin{aligned}
B_2(0)&=
\vmt{
lc(0)&0&md(0)\\
mg_2(0)&0&e(0)\\
0&1&0}=(-lce + m^2 d g_2)(0)\\
&=
-\dfrac{b_{m1}(0)}{(m-1)!^2}+\dfrac{l(m+1) b_{m2}(0) c(0)}{m}.
\end{aligned}
\]
If $2m-l-1<0$, then $m<l$, and \eqref{eq:tgt0100} is
$t^{m-1}B_3(t)$, where 
\[
\begin{aligned}
B_3(0)&=
\vmt{
0&0&md(0)\\
mg_2(0)&0&e(0)\\
0&1&0}=m^2d(0)g_2(0)=
-\dfrac{b_{m1}(0)}{(m-1)!^2}\,.
\end{aligned}
\]
This show the assertion for $\tau_g$.
\medskip

\noindent (2) and (3). We assume $l < m$ and we shall use 
the same notation of  case (1).
Setting $\nu_2(t)=\tilde\nu_2(\gamma(t))$, we can show that $ \nu_2(t)=(t^l d(t),te(t),1)$, where
\small
\[
\begin{aligned}
d(t)&=\dfrac{1}{2 m m!}
 \bigg(-2 m t^{m-l} b_{m1} + 
   2 m m!  t^{l}
     a b_{m1} c^2
- 2 mm!  b_0 c   +
   2 m! t
     a b_{m2} c  +  2 m m! t
     a b_{m2} c   \\ 
        &\hspace{20mm}+ 
   m m! t^{2 l}
     b_{m1} c^3  a' +  m!  t^{l+1}
     b_{m2} c^2 a' +
     m  m! t^{l+1}
     b_{m2} c^2 a' - m t^{l} c^2 m! b_0' \\
   &\hspace{20mm}+ 
 2 m!  t^{2} a c  
b_{m2,v} +
  m!  t^{l+2} c^2  a' 
b_{m2,v} - 2 m t^m c 
b'_{m1}- 2 m t^{m-l+1} 
b_{m2,u}\bigg)\end{aligned}
\]
\normalsize
and $e$ is the same as in \eqref{eq:d0100}.
We assume $l\leq m/2$.
Then $\hat\gamma = t^l \tilde\rho$, where $\tilde\rho(t)=(c(t), t^lg_2(t)$, $t^lg_3(t))$ and
\[\begin{aligned}
g_2(t)&=
\dfrac{2 t^{m-2l} + m!a(t) c(t)^2 }{2 m!},
\\
g_3(t)&=
\dfrac{2 t^{m-l} b_{m1}(t) c(t) +2 t^{m-2l+1} b_{m2}(\gamma(t)) 
+m! b_0(t) c(t)^2 }{2  m!}\, .
\end{aligned}
\]

Then $\hat \gamma' = t^l\rho$, where 
$\rho = l\tilde\rho+t\tilde\rho'$, with $\rho(0)\ne0$. Since $\hat \gamma$ has multiplicity $l$, we need replace $m-1$ in equations \eqref{eq:reg:normaliz:curv} by $l-1$.
Here, we see
\[
\begin{aligned}
&g_2(0)=a(0)c(0)^2/2\  (\text{if } l<m/2),\ 
g_2(0)=a(0)c(0)^2/2+1/m!\  (\text{if }m=2l),\\
&g_3(0)=b_0(0) c(0)^2/2,\ 
e(0)=-(1+m) b_{m2}(0)/m\,.
\end{aligned}
\]
To see the order, we may use $\nu_2(t)$ instead of $\nu\circ\gamma(t)$
in \eqref{eq:reg:normaliz:curv}.
We see
\begin{equation}\label{eq:rho1000}
\begin{aligned}
\rho&=\Big(lc+tc',t^l(2lg_2+tg_2'),t^l(2lg_3+tg_3')\Big),\\
\rho'&=\Big((l+1)c'+tO(1),t^{l-1}(2l^2g_2+tO(1)),t^{l-1}(2l^2g_3+tO(1))\Big).
\end{aligned}
\end{equation}
By applying the formula
\[
\vmt{
x_{11}&x_{12}&x_{13}\\
kx_{21}&x_{22}&x_{23}\\
kx_{31}&x_{32}&x_{33}}
=
\vmt{
x_{11}&kx_{12}&kx_{13}\\
x_{21}&x_{22}&x_{23}\\
x_{31}&x_{32}&x_{33}},
\]
for $k=t^{l-1}$,
we see $|\rho,\rho',\nu_2|$ is
\begin{align}
 &\vmt{
lc+tO(1)&(l+1)c'+tO(1)&t^ld\\
t^l(2lg_2+tO(1))&t^{l-1}(2l^2g_2+tO(1))&te\\
t^l(2lg_3+tO(1))&t^{l-1}(2l^2g_3+tO(1))&1}\nonumber\\
=&
\vmt{
lc+tO(1)&t^{l-1}(l+1)c'+tO(1)&t^{2l-1}d\\
t(2lg_2+tO(1))&t^{l-1}(2l^2g_2+tO(1))&te\\
t(2lg_3+tO(1))&t^{l-1}(2l^2g_3+tO(1))&1}\label{eq:kgt1100}\\
=&t^{l-1}C_1(t)\quad
\left(C_1(t)=
\vmt{
lc+tO(1)&(l+1)c'+tO(1)&t^{2l-1}d\\
t(2lg_2+tO(1))&2l^2g_2+tO(1)&te\\
t(2lg_3+tO(1))&2l^2g_3+tO(1)&1}\right).\nonumber
\end{align}
Then $C_1(0)=2l^3c(0)g_2(0)$. 
This shows the assertion for $\kappa_g$.
By \eqref{eq:rho1000}, we see $\inner{\rho'}{\nu_2} = t^{l-1}(2l^2g_3+tO(1))$
and
$|\rho,\nu_2,\nu_2'|(0)=l(m+1)c(0)b_{m2}/m$.
This shows the assertions for $\kappa_n$ and $\tau_g$.

Next we assume $l > m/2$.
In this case, $\hat\gamma=t^l (c(t), t^{m-l}g_2(t), t^{m-l+1}g_3(t))$.
We set $\tilde\rho(t)=(c(t), t^{m-l}g_2(t), t^{m-l+1}g_3(t))$ and
\[
\begin{aligned}
g_2(t)&=
\dfrac{2+m!\,t^{2l-m}a(t) c(t)^2 }{2m!},
\\
g_3(t)&=
\dfrac{2 t^{l-1} b_{m1}(t) c(t) +2 b_{m2}(\gamma(t)) 
+ m!\,t^{2l-m-1}b_0(t) c(t)^2}{2 m!}.
\end{aligned}
\]
Here, it holds that
\[
\begin{aligned}
&g_2(0)=1/m!,\ 
e(0)=-(1+m) b_{m2}(0)/m,\\
&g_3(0)=b_{m2}(0)/m!\ (\text{if }2l-m-1>0),\\
&g_3(0)=b_0(0)c(0)^2/2+b_{m2}(0)/m!\ (\text{if }m=2l-1).
\end{aligned}
\]
It holds that $\hat \gamma'= t^{l-1}\rho$, with $\rho=l\tilde\rho+t\tilde\rho'$ and 
 $\rho(0)\ne0$.
We see
\begin{equation}\label{eq:rho10000}
\begin{aligned}
\rho(t)&=\Big(lc+tc',t^{m-l}(mg_2+tg_2'),t^{m-l+1}((m+1)g_3+tg_3')\Big),\\
\rho'(t)&=\Big(
(l+1)c'+tO(1),t^{m-l-1}(m(m-l)g_2+tO(1)),\\
&\hspace{50mm}
t^{m-l}((m+1)(m-l+1)g_3+tO(1))
\Big).
\end{aligned}
\end{equation}
By the similar method to \eqref{eq:kgt1100},
we see $|\rho,\rho',\nu_2|$ is
\[\begin{aligned}
&\vmt{
lc+tO(1)&(1+l)c'+tO(1)&t^l d\\
t^{m-l}(m g_2+tO(1))&t^{m-l-1} (m (m-l) g_2+tO(1))&t e\\
t^{m-l+1}((m+1) g_3+tO(1))&t^{m-l} ((m+1)(m-l+1) g_3+tO(1))&1}\\
=&
\vmt{
lc+tO(1)&t^{m-l-1}((1+l)c'+tO(1))&t^{m-1} d\\
t(m g_2+tO(1))&t^{m-l-1} (m (m-l) g_2+tO(1))&t e\\
t^2((m+1) g_3+tO(1))&t^{m-l} ((m+1)(m-l+1) g_3+tO(1))&1}\\
=&t^{m-l-1}C_2(t),\\
C_2(t)=&\vmt{
lc+tO(1)&(1+l)c'+tO(1)&t^{m-1} d\\
t(m g_2+tO(1))&m (m-l) g_2+tO(1)&t e\\
t^2((m+1) g_3+tO(1))&t((m+1)(m-l+1) g_3+tO(1))&1}.
\end{aligned}
\]
Then 
$
C_2(0)=
l(m-l) m c(0) g_2(0)=
l(m-l)c(0)/(m-1)!$
and, replacing $m-1$ by $l-1$ in equations \eqref{eq:reg:normaliz:curv}, 
this shows the assertion for $\kappa_g$.
By \eqref{eq:rho1000}, we see 
$\inner{\rho'}{\nu_2} = t^{m-l}C_3(t)$, where $C_3(t) = m(m-l)g_2(t)e_2(t) + (m+1)(m-l+1)g_3(t) + tO(1)$.  It holds that 
$$
C_3(0)
=\begin{dcases}
\dfrac{(m+1)b_{m2}(0)}{m!}\ &(m<2l-1),\\
(m+1)\left(\dfrac{l b_0(0)c(0)^2}{2(l!)^2}+\dfrac{b_{m2}(0)}{m!}\right) &(m=2l-1).
\end{dcases}
$$
and
$|\rho,\nu_2,\nu_2'|(0)=-lc(0)e(0)
=l (m+1)c(0)b_{m2}(0)/m$.
This shows the assertions for $\kappa_n$ and $\tau_g$.
\end{proof}

In particular, we have the following corollary
on boundedness directly obtained from Theorem \ref{thm:ord1}.

\begin{corollary}\label{cor:bdd}
Let $f:(\R^2,0)\to (\R^3,0)$ be an  $m$-type edge with $m\geq 2$, 
and $\gamma:(\R,0)\to(\R^2,0)$ be a regular curve with order of contact $l\geq2$ with the null direction of $f$ at $0$. 
\begin{enumerate}
\item The case $l\geq m$.
For $\kappa_g$, 
\begin{itemize}
\item if $l\geq 2m$, then $\kappa_g$ is bounded at $0$;
\item if $m< l< 2m$, then $\kappa_g$ is unbounded at $0$;
\item if $m= l$ and $\tilde\kappa^{(l-1)}(0)\ne0$, 
then $\kappa_g$ is unbounded at $0$.
\end{itemize}
For $\kappa_n$, if $\omega_{m,m+1}(0)\ne0$, then $\kappa_n$ is unbounded at $0$.
For $\tau_g$, 
\begin{itemize}
\item if 
$m\leq l<2m-1$ and $\omega_{m,m+1}(0)\ne0$, then $\tau_g$ is unbounded at $0$;
\item if $l=2m-1$ and 
$m(l-1)!\kappa_t(0)
+{(m-1)!^2}\,\tilde\kappa^{(l-2)}(0)\, \omega_{m,m+1}(0)\ne0$, 
then $\tau_g$ is bounded at $0$;
\item if $l>2m-1$, then $\tau_g$ is bounded at $0$.
\end{itemize}

\item
The case $m/2<l<m$.
In this case, $\kappa_g$ is unbounded at $0$.
If $l=(m+1)/2$, then $\kappa_n$ is bounded at $0$.
If $m>l>(m+1)/2$ and $\omega_{m,m+1}(0)\ne0$, then $\kappa_n$ is 
unbounded at $0$.
If $\omega_{m,m+1}(0)\ne0$, then $\tau_g$ is 
unbounded at $0$.

\item
The case $l\leq m/2$.
In this case, $\kappa_g$ and $\kappa_n$ are bounded at $0$.
If $\omega_{m,m+1}(0)\ne0$, then $\tau_g$ is unbounded at $0$.
\end{enumerate}
\end{corollary}

We consider the case that 
$f\colon(\R^2,0)\to(\R^3,0)$ is
a cuspidal edge. 
By definition, it is a $(2,3)$-edge, in particular, a $2$-type edge. 
Then by Theorem \ref{thm:ord1}, 
the following assertion holds.

\begin{corollary}\label{thm:ordce}
Let  $f:(\R^2,0) \to (\R^3, 0)$ be a cuspidal edge,
and let $\gamma:(\R,0)\to (\R^2, 0)$  
be a regular curve with order of contact $l \geq 2$ with 
the null direction of $f$ at $0$ and $\tilde\kappa$ 
the curvature of $\gamma$ written in the normal form of $f$.
Then, it holds that:

For $\kappa_g$, 
	\begin{itemize}
	\item if $l=2$,
	then $\ord\kappa_g\geq-1$, and
	$\ord\kappa_g=-1$ if and only if
	$\tilde\kappa^{(l-1)}(0)\ne0$.
	\item if $l=3$ or $4$, then
	$\ord\kappa_g= l-4$;
	\item if $l\geq5$, then $\ord\kappa_g\geq 1$,
	and $\ord\kappa_g=1$ is equivalent to
	\[
	\begin{dcases}
	(l-1)!\kappa_t(0) \omega_{2,3}(0)-12\tilde\kappa^{(l-2)}(0)\ne0 & 
(\text{if  $l=5$}),\\ 
	\kappa_t(0)\omega_{2,3}(0)\ne0 & (\text{if $l>5$});
	\end{dcases}
	\]
	\end{itemize}

For $\kappa_n$,
	it holds that $\ord\kappa_n=-1$.

For $\tau_g$, 
	\begin{itemize}
	\item if $l=2$ or $3$, then 
	$\ord\tau_g\geq l-3$, and
	$\ord\tau_g=l-3$ is equivalent to 
	\[\begin{dcases}
	\omega_{2,3}(0)\ne0 & (\text{if $l< 3$}),\\
	2(l-1)!\kappa_t(0)+\tilde\kappa^{(l-2)}(0)\, \omega_{2,3}(0)\ne0& 
	(\text{if $l= 3$});
	\end{dcases}
	\]
	\item if $l\geq4$, then 
	$\ord\tau_g\geq 0$, and
	$\ord\tau_g=0$ if and only if $\kappa_t(0)\ne0$.
	\end{itemize}
\end{corollary}
\begin{proof}
Since $\omega_{2,3}$ corresponds to the cuspidal curvature $\kappa_c$
and it does not vanish at $0$ (\cite[Proposition 3.11]{msuy}),
we have the assertion by Theorem \ref{thm:ord1}.
\end{proof}
About the boundedness, we have the following corollary 
immideately from Theorem \ref{thm:ordce}.
\begin{corollary}\label{cor:ce}
Under the same assumption of Corollary \ref{thm:ordce},
we have the following:
	\begin{enumerate}
	\item For the geodesic curvature $\kappa_g$, 
	\begin{itemize}
	\item if $l\geq 4$, then $\kappa_g$ is bounded at $0$;
	\item if $l=3$, then $\kappa_g$ is unbounded at $0$;
	\item if $l=2$ and $\tilde\kappa'(0)\ne0$, 
	then $\kappa_g$ is unbounded at $0$.
	\end{itemize}
	\item The normal curvature $\kappa_n$ is unbounded at $0$.
	\item For the geodesic torsion $\tau_g$,
	\begin{itemize}
		\item if $l=2$, then $\tau_g$ is unbounded at $0$;
	\item if $l=3$ and 
{$4 \kappa_t(0)
	+\,\tilde\kappa'(0)\, \kappa_c(0)\ne0$}, 
	then $\tau_g$ is bounded at $0$;
	\item if $l\geq 4$, then $\tau_g$ is bounded at $0$,
	\end{itemize}
where $\kappa_c$ is the cuspidal curvature (cf. \cite{msuy}) corresponding to $\omega_{2,3}$.
	\end{enumerate}
\end{corollary}
Note that $\ord \kappa_g\geq-1$ for $l\geq2$ 
is pointed out in \cite[Proposition 2.19]{dz}. 

We observe that although in the above results we could not guarantee 
that the three invariants are bounded at the same time near a singular point, 
it is easy to find an example where it happens: 
taking $f=(u,\frac{v^2}{2},v^5)$ and $\gamma(t)=(t^4,t)$, it holds that 
$m=2, l=4$, $\ord \kappa_g = 0, \ord \kappa_n = 1, \ord \tau_g = 3$  
(see Figure \ref{fig:curvatures1}). 
Thus these three invariants are bounded at $0$ (cf. Corollary \ref{cor:bdd}).
For the cuspidal edge $f(u,v)=(u,v^2,v^3)$ and the same $\gamma$,
we see that $\kappa_g$ and $\tau_g$ are bounded, 
but $\kappa_n$ is unbounded at $0$ (cf. Corollary \ref{cor:ce}). 
Figure \ref{fig:curvatures2} shows the graphs of these invariants near $0$.

\begin{figure}[ht]
	\centering
	\includegraphics[width=.3\linewidth] {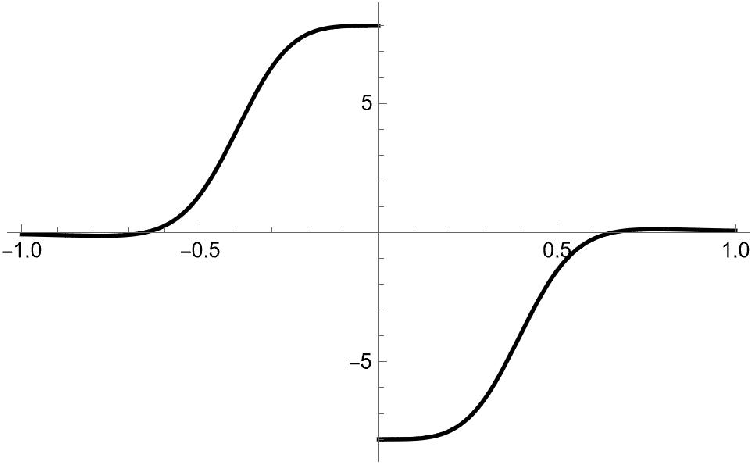}
	\hspace{3mm}
	\includegraphics[width=.3\linewidth]
	{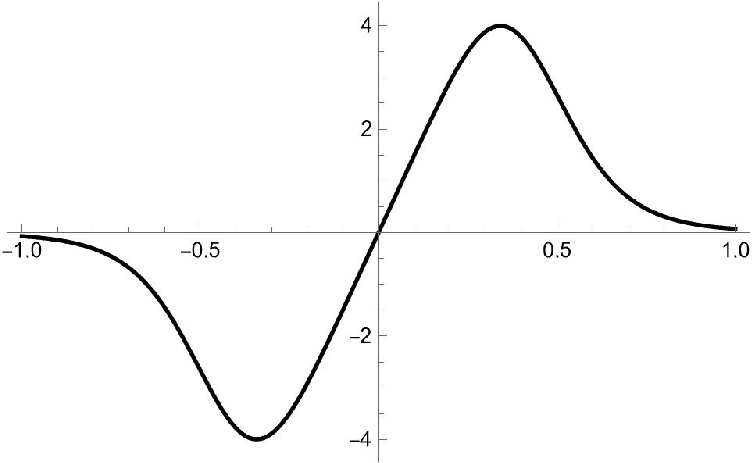}
	\hspace{3mm}
	\includegraphics[width=.3\linewidth]
	{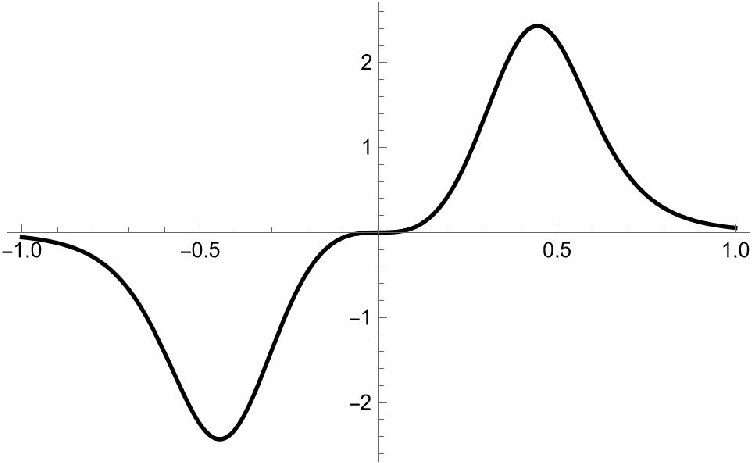}
	\caption{The graphs of $\kappa_g$ (left), $\kappa_n$ (middle) 
		and $\tau_g$ (right) of the curve $\hat{\gamma}(t)=f(\gamma(t))$, 
		where $f=(u,\frac{v^2}{2},v^5)$ and $\gamma(t)=(t^4,t)$.} 
	\label{fig:curvatures1}
\end{figure}

\begin{figure}[ht]
	\centering
	\includegraphics[width=.3\linewidth] {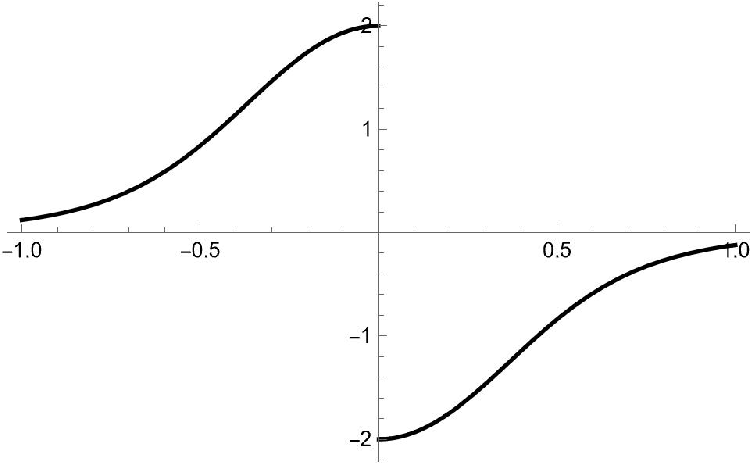}
	\hspace{3mm}
	\includegraphics[width=.3\linewidth]
	{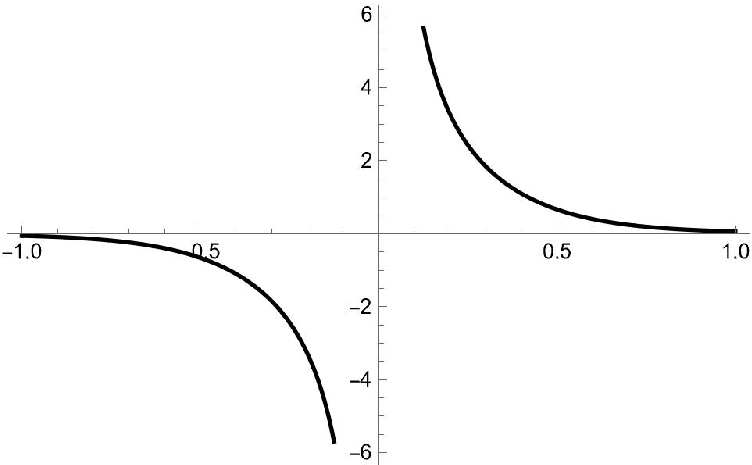}
	\hspace{3mm}
		\includegraphics[width=.3\linewidth]
	{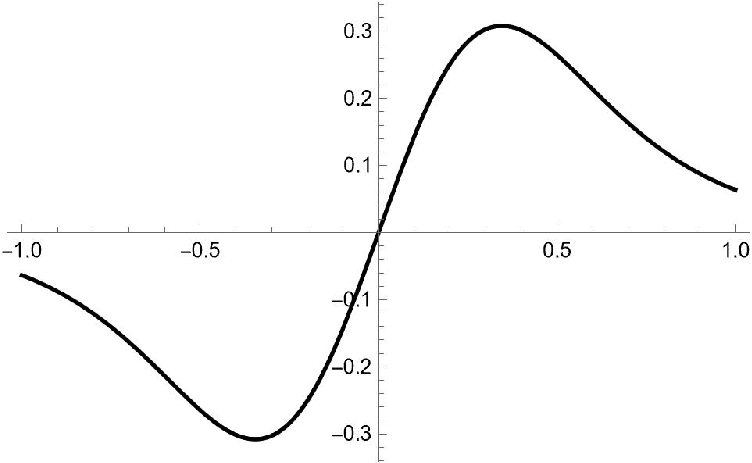}
	\caption{The graphs of $\kappa_g$ (left), $\kappa_n$ (middle) 
		and $\tau_g$ (right) of the curve $\hat{\gamma}(t)=f(\gamma(t))$, 
	where $f(u,v)=(u,v^2,v^3)$ and $\gamma(t)=(t^4,t)$.} 
	\label{fig:curvatures2}
\end{figure}

\subsection*{Acknowledgements}
The authors thank the referee,
Yuki Hattori and Atsufumi Honda for valuable comments and suggestions.

\appendix
\section{Generalized biases for plane curve}\label{sec:genbias}

Let $\gamma:(\R,0)\to(\R^2,0)$ be a curve-germ of $(m,n)$-type
which is given by the form \eqref{eq:mnnormal} in the
$xy$-plane $(\R^2,0)$.
The terms $a_i$ $(i=2,\ldots,\lfloor n/m\rfloor)$
measures the bias of $\gamma$ near singular point.
We call $a_{i+1}$ the $(m,im)$-{\it bias\/} 
$(i=2,\ldots,\lfloor n/m\rfloor)$
of $\gamma$ at $0$,
and 
it is denoted by $\beta_{m,im}$.
We call $b(0)$ is called the 
$(m,n)$-{\it cuspidal curvature\/} as in \cite{honda-saji},
and it is denoted by $r_{m,n}$.

If $m$ and $n$ are even, then it is a half part of
a curve of $(m/2,n/2)$-type,
we consider the following cases:
(1) both $m,n$ are odd, (2) $m$ is odd and $n$ is even,
and
(3) $m$ is even and $n$ is odd.
Moreover, let $a_k$ denotes the first non-zero 
term of $a_i$ $(i=2,\ldots,\lfloor n/m\rfloor)$.
We consider the case $(1)$ and $(2)$.
Then $\gamma$ passes through the origin tangent to the
$x$-axis. 
In the case $(1)$, 
if $k$ is odd, it also passes across the $x$-axis.
If $k$ is even, it approaches to the origin
from one side of $x$-axis and goes away into the same side
of $x$-axis, and
if there does not exist such $k$ (namely, the bias is zero),
it passes through the $x$-axis.
In the case $(2)$, 
if  the bias is zero,
it approaches to the origin
from one side of $x$-axis and goes away into the same side
of $x$-axis.
Figure \ref{fig:bias1} shows the images of the curves
$\gamma_1:t\mapsto (t^3, a_1 t^6 + a_2 t^9 + t^{11})$
with $(a_1,a_2)=(1,0),(0,1),(0,0)$ from left to right.
Figure \ref{fig:bias2} shows the images of the curves
$\gamma_2:t\mapsto (t^3, a_1 t^6 + a_2 t^9 + t^{14})$
with $(a_1,a_2)=(1,0),(0,1),(0,0)$ from left to right.

\begin{figure}[ht]
\centering
\includegraphics[width=.3\linewidth] {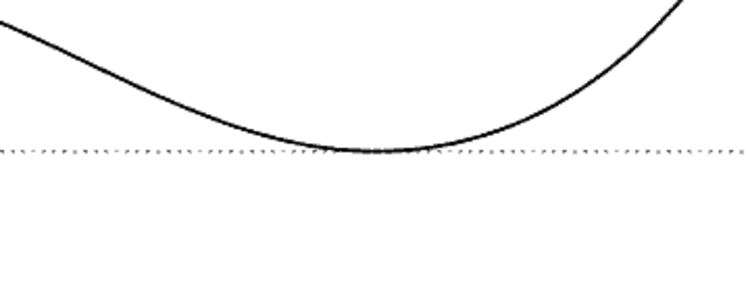}
\hspace{3mm}
\includegraphics[width=.3\linewidth]
{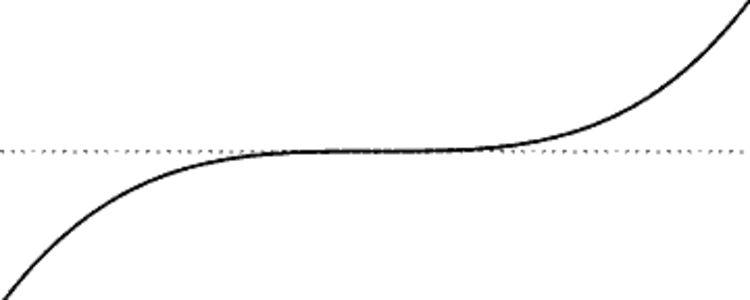}
\hspace{3mm}
\includegraphics[width=.3\linewidth]
{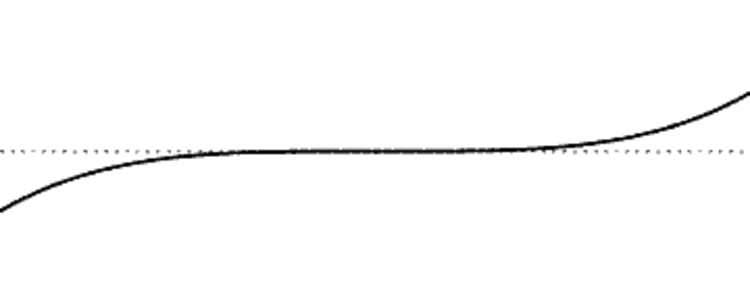}
\caption{The images of the curves $\gamma_1$} 
\label{fig:bias1}
\end{figure}

\begin{figure}[ht]
\centering
\includegraphics[width=.3\linewidth]
{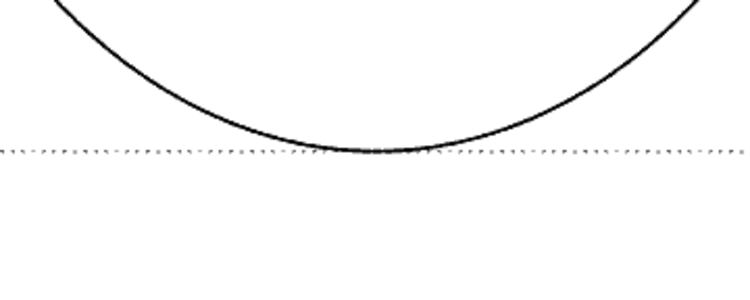}
\hspace{3mm}
\includegraphics[width=.3\linewidth]
{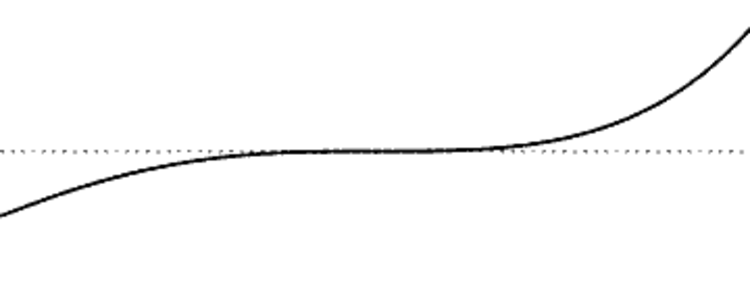}
\hspace{3mm}
\includegraphics[width=.3\linewidth]
{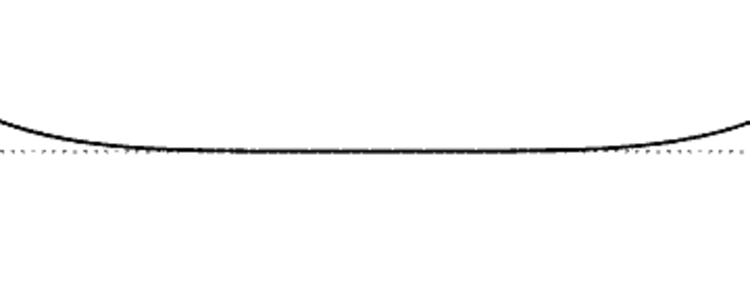}
\caption{The images of the curves $\gamma_2$} 
\label{fig:bias2}
\end{figure}

We consider the case $(3)$.
Then $\gamma$ approaches the origin from a direction of
the $x$-axis and making a cusp, and back to the same direction.
If $k$ is both odd and even, 
it approaches to the origin
from one side of $x$-axis and goes away into the same side
of $x$-axis.
If  the bias is zero,
it passes through the $x$-axis.
Figure \ref{fig:bias3} shows the images of the curves
$\gamma_3:t\mapsto (t^4, a_1 t^8 + a_2 t^{12} + t^{13})$
with $(a_1,a_2)=(1,0),(0,1),(0,0)$ from left to right.
\begin{figure}[ht]
\centering
\includegraphics[width=.3\linewidth]
{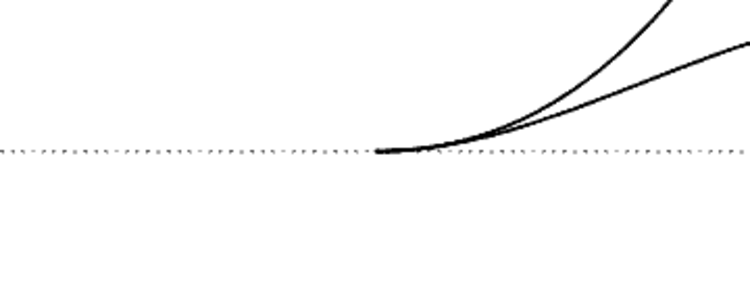}
\hspace{3mm}
\includegraphics[width=.3\linewidth]
{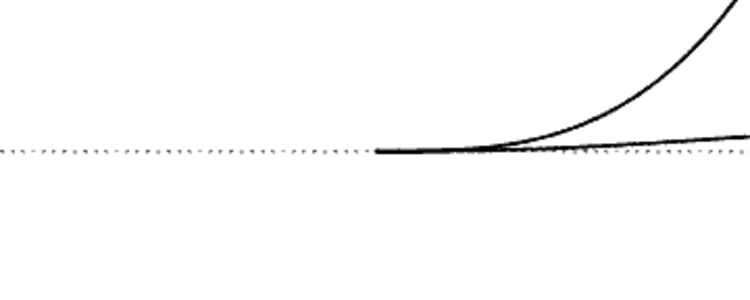}
\hspace{3mm}
\includegraphics[width=.3\linewidth]
{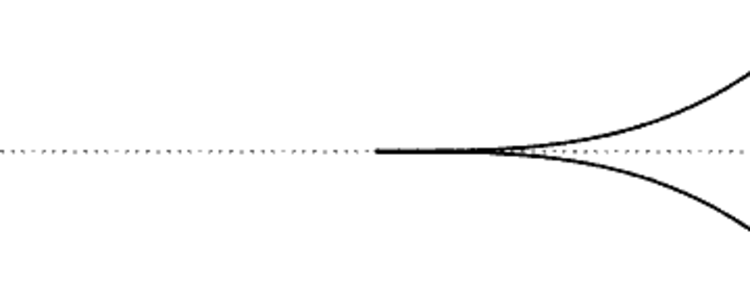}
\caption{The images of the curves $\gamma_3$} 
\label{fig:bias3}
\end{figure}

\begin{example}\label{ex:gamma3}
Let $\gamma$ be a curve-germ $\A^{3}$-equivalent to $(t^3,0)$.
%We set
%\begin{equation}\label{eq:gamma0}
%\gamma(t)=
%\left(
%\sum_{i=3}^{10}
%\dfrac{a_i}{i!}t^i,\ 
%\sum_{i=3}^{10}
%\dfrac{b_i}{i!}t^i\right)+O(10)\quad((a_3,b_3)\ne(0,0)),
%\end{equation}
%where $O(10)$ stands for the terms whose degrees are
%greater than $10$.
We set
$$
\tilde a_i=
\dfrac{\gamma^{(3)}(0)\cdot\gamma^{(i)}(0)}{i!|\gamma^{(3)}(0)|},\quad
\tilde b_i=
\dfrac{\det(\gamma^{(3)}(0),\gamma^{(i)}(0))}{i!|\gamma^{(3)}(0)|}.
$$
One can calculate the invariants up to $10$ degrees as follows.
The $(3,4)$-cuspidal curvature $r_{3,4}$ is
\begin{equation}\label{eq:delta34}
r_{3,4}=\dfrac{\tilde{b}_4}{\tilde{a}_3^{4/3}}.
\end{equation}
If $r_{3,4}\ne0$, i.e., $\tilde b_4\ne0$, then 
$\gamma$ is $\A$-equivalent to $(t^3,t^4)$.
We assume $\tilde b_4=0$.
Then the
$(3,5)$-cuspidal curvature $r_{3,5}$ is
\begin{equation}\label{eq:delta35}
r_{3,5}=
\dfrac{\tilde{b}_5}{\tilde{a}_3^{5/3}}.
\end{equation}
If $r_{3,5}\ne0$, i.e., $\tilde b_5\ne0$, then 
$\gamma$ is $\A$-equivalent to $(t^3,t^5)$.
We assume $\tilde b_5=0$.
Then the $(3,6)$-bias $\beta_{3,6}$ and the
$(3,7)$-cuspidal curvature $r_{3,7}$ are
\begin{align}
\label{eq:beta36}
\beta_{3,6}&=
\dfrac{\tilde{b}_6}{\tilde{a}_3^2},\\
\label{eq:delta37}
r_{3,7}
&=
\dfrac{-7 \tilde{a}_4 \tilde{b}_6+2 \tilde{a}_3 \tilde{b}_7}{2 \tilde{a}_3^{10/3}}.
\end{align}
If $r_{3,7}\ne0$, i.e., $-7 \tilde{a}_4 \tilde{b}_6+2 \tilde{a}_3 \tilde{b}_7\ne0$, then 
$\gamma$ is $\A^7$-equivalent to $(t^3,t^7)$.
We assume $r_{3,7}=0$, i.e.,
$\tilde{b}_7=7 \tilde{a}_4 \tilde{b}_6/(2 \tilde{a}_3)$.
Then the
$(3,8)$-cuspidal curvature $r_{3,8}$ is
\begin{equation}\label{eq:delta38}
r_{3,8}
=
\dfrac{-35 \tilde{a}_4^2 \tilde{b}_6+2 \tilde{a}_3 (-28 \tilde{a}_5 \tilde{b}_6+5 \tilde{a}_3 \tilde{b}_8)}
{10 \tilde{a}_3^{14/3}}.
\end{equation}
If $r_{3,8}\ne0$, then
$\gamma$ is $\A^8$-equivalent to $(t^3,t^8)$.
We assume $r_{3,8}=0$, i.e.,
$\tilde{b}_8=7 (5 \tilde{a}_4^2+8 \tilde{a}_3 \tilde{a}_5) \tilde{b}_6/(10 \tilde{a}_3^2)$.
Then the $(3,9)$-bias $\beta_{3,9}$ and
the $(3,10)$-cuspidal curvature $r_{3,10}$
are
\begin{align}
\beta_{3,9}
&=-
\dfrac{
63 \tilde{a}_4 \tilde{a}_5 \tilde{b}_6+42 \tilde{a}_3 \tilde{a}_6 \tilde{b}_6-5 \tilde{a}_3^2 \tilde{b}_9}{5 \tilde{a}_3^5},
\label{eq:beta39}\\
r_{3,10}
&=\dfrac{
(10 \tilde{a}_3^3 \tilde{b}_{10}+945 \tilde{a}_4^2 \tilde{a}_5 \tilde{b}_6-42 \tilde{a}_3 (3 \tilde{a}_5^2-10 \tilde{a}_4 \tilde{a}_6) \tilde{b}_6-15 \tilde{a}_3^2 (8 \tilde{a}_7 \tilde{b}_6+5 \tilde{a}_4 \tilde{b}_9))}
{10 \tilde{a}_3^{19/3}}.
\label{eq:delta310}
\end{align}
\end{example}
\begin{proof}[Proof of Example \ref{ex:gamma3}]
%We assume that $\gamma$ is given by \eqref{eq:gamma0}.
By rotating $\gamma$ in $\R^3$, we can write
\begin{equation}\label{eq:gammat}
\gamma(t)=
\left(
\sum_{i=3}^{10}
\dfrac{\tilde a_i}{i!}t^i,\ 
\sum_{i=4}^{10}
\dfrac{\tilde b_i}{i!}t^i\right)+O(10).
\end{equation}
We set 
\[
\phi(t)=t\left(6\sum_{i=3}^{10}\dfrac{\tilde a_i}{i!}t^{i-3}\right)^{1/3},
\]
and the inverse function of $s=\phi(t)$ as $t=\psi(s)$.
We set $\psi(s)=\sum_{i=1}^{10}\psi_is^i/i!+O(10)$.
Then we have
\small
\[
\begin{aligned}
\psi_1&=1/\tilde{a}_3^{1/3},\\
\psi_2&=-\tilde{a}_4/(6 \tilde{a}_3^{5/3}),\\
\psi_3&=(5 \tilde{a}_4^2-4 \tilde{a}_3 \tilde{a}_5)/(40 \tilde{a}_3^3),\\
\psi_4&=(-175 \tilde{a}_4^3+252 \tilde{a}_3 \tilde{a}_4 \tilde{a}_5-72 \tilde{a}_3^2 \tilde{a}_6)/(1080 \tilde{a}_3^{13/3}),\\
\psi_5&=(13475 \tilde{a}_4^4-27720 \tilde{a}_3 \tilde{a}_4^2 \tilde{a}_5+10080 \tilde{a}_3^2 \tilde{a}_4 \tilde{a}_6+432 \tilde{a}_3^2 (14 \tilde{a}_5^2-5 \tilde{a}_3 \tilde{a}_7))/(45360 \tilde{a}_3^{17/3}),\\
\psi_6&=(-1575 \tilde{a}_4^5+4200 \tilde{a}_3 \tilde{a}_4^3 \tilde{a}_5-1680 \tilde{a}_3^2 \tilde{a}_4^2 \tilde{a}_6+96 \tilde{a}_3^2 \tilde{a}_4 (-21 \tilde{a}_5^2+5 \tilde{a}_3 \tilde{a}_7)\\
&\hspace{10mm}
+16 \tilde{a}_3^3 (42 \tilde{a}_5 \tilde{a}_6-5 \tilde{a}_3 \tilde{a}_8))/(2240 \tilde{a}_3^7),\\
\psi_7&=
(475475 \tilde{a}_4^6-1556100 \tilde{a}_3 \tilde{a}_4^4 \tilde{a}_5+655200 \tilde{a}_3^2 \tilde{a}_4^3 \tilde{a}_6-42120 \tilde{a}_3^2 \tilde{a}_4^2 (-28 \tilde{a}_5^2+5 \tilde{a}_3 \tilde{a}_7)\\
&\hspace{10mm}
+3240 \tilde{a}_3^3 \tilde{a}_4 (-182 \tilde{a}_5 \tilde{a}_6+15 \tilde{a}_3 \tilde{a}_8)\\
&\hspace{10mm}
-1296 \tilde{a}_3^3 (91 \tilde{a}_5^3-60 \tilde{a}_3 \tilde{a}_5 \tilde{a}_7+5 \tilde{a}_3 (-7 \tilde{a}_6^2+\tilde{a}_3 \tilde{a}_9)))/(233280 \tilde{a}_3^{25/3}),\\
\psi_8&=
(-155520 \tilde{a}_{10} \tilde{a}_3^6+11 (-4447625 \tilde{a}_4^7+17243100 \tilde{a}_3 \tilde{a}_4^5 \tilde{a}_5-7497000 \tilde{a}_3^2 \tilde{a}_4^4 \tilde{a}_6\\
&\hspace{10mm}
+2570400 \tilde{a}_3^2 \tilde{a}_4^3 (-7 \tilde{a}_5^2+\tilde{a}_3 \tilde{a}_7)-45360 \tilde{a}_3^3 \tilde{a}_4^2 (-238 \tilde{a}_5 \tilde{a}_6+15 \tilde{a}_3 \tilde{a}_8)\\
&\hspace{10mm}+15552 \tilde{a}_3^4 (-98 \tilde{a}_5^2 \tilde{a}_6+20 \tilde{a}_3 \tilde{a}_6 \tilde{a}_7+15 \tilde{a}_3 \tilde{a}_5 \tilde{a}_8)\\
&\hspace{10mm}+5184 \tilde{a}_3^3 \tilde{a}_4 (833 \tilde{a}_5^3-420 \tilde{a}_3 \tilde{a}_5 \tilde{a}_7+5 \tilde{a}_3 (-49 \tilde{a}_6^2+5 \tilde{a}_3 \tilde{a}_9))))
/(6998400 \tilde{a}_3^{29/3}),\\
\psi_9&=
(17920\tilde{a}_{10} \tilde{a}_4\tilde{a}_3^6
+
2480625 \tilde{a}_4^8-11113200 \tilde{a}_3 \tilde{a}_4^6 \tilde{a}_5+4939200 \tilde{a}_3^2 \tilde{a}_4^4 (3 \tilde{a}_5^2+\tilde{a}_4 \tilde{a}_6)\\
&\hspace{10mm}-70560 \tilde{a}_3^3 \tilde{a}_4^2 (84 \tilde{a}_5^3+140 \tilde{a}_4 \tilde{a}_5 \tilde{a}_6+25 \tilde{a}_4^2 \tilde{a}_7)+4032 \tilde{a}_3^4 (84 \tilde{a}_5^4+840 \tilde{a}_4 \tilde{a}_5^2 \tilde{a}_6\\
&\hspace{10mm}+600 \tilde{a}_4^2 \tilde{a}_5 \tilde{a}_7+25 \tilde{a}_4^2 (14 \tilde{a}_6^2+5 \tilde{a}_4 \tilde{a}_8))+2560 \tilde{a}_3^6 (12 \tilde{a}_7^2+21 \tilde{a}_6 \tilde{a}_8+14 \tilde{a}_5 \tilde{a}_9)\\
&\hspace{10mm}-4480 \tilde{a}_3^5 (72 \tilde{a}_5^2 \tilde{a}_7+\tilde{a}_5 (84 \tilde{a}_6^2+90 \tilde{a}_4 \tilde{a}_8)+5 \tilde{a}_4 (24 \tilde{a}_6 \tilde{a}_7+5 \tilde{a}_4 \tilde{a}_9)))/(89600 \tilde{a}_3^{11}),\\
\psi_{10}&=
13 (-16865646875 \tilde{a}_4^9+85717170000 \tilde{a}_3 \tilde{a}_4^7 \tilde{a}_5+19595520 \tilde{a}_{10} \tilde{a}_3^6 (-10 \tilde{a}_4^2+3 \tilde{a}_3 \tilde{a}_5)\\
&\hspace{10mm}-38710980000 \tilde{a}_3^2 \tilde{a}_4^6 \tilde{a}_6+2844072000 \tilde{a}_3^2 \tilde{a}_4^5 (-49 \tilde{a}_5^2+5 \tilde{a}_3 \tilde{a}_7)\\
&\hspace{10mm}-1422036000 \tilde{a}_3^3 \tilde{a}_4^4 (-70 \tilde{a}_5 \tilde{a}_6+3 \tilde{a}_3 \tilde{a}_8)\\
&\hspace{10mm}+372314880 \tilde{a}_3^4 \tilde{a}_4^2 (-154 \tilde{a}_5^2 \tilde{a}_6+20 \tilde{a}_3 \tilde{a}_6 \tilde{a}_7+15 \tilde{a}_3 \tilde{a}_5 \tilde{a}_8)\\
&\hspace{10mm}+206841600 \tilde{a}_3^3 \tilde{a}_4^3 (385 \tilde{a}_5^3-132 \tilde{a}_3 \tilde{a}_5 \tilde{a}_7+\tilde{a}_3 (-77 \tilde{a}_6^2+5 \tilde{a}_3 \tilde{a}_9))\\
&\hspace{10mm}-1119744 \tilde{a}_3^4 \tilde{a}_4 (10241 \tilde{a}_5^4-7980 \tilde{a}_3 \tilde{a}_5^2 \tilde{a}_7+150 \tilde{a}_3^2 (4 \tilde{a}_7^2+7 \tilde{a}_6 \tilde{a}_8)\\
&\hspace{10mm}+70 \tilde{a}_3 \tilde{a}_5 (-133 \tilde{a}_6^2+10 \tilde{a}_3 \tilde{a}_9))+186624 \tilde{a}_3^5 (22344 \tilde{a}_5^3 \tilde{a}_6-10080 \tilde{a}_3 \tilde{a}_5 \tilde{a}_6 \tilde{a}_7\\
&\hspace{10mm}-3780 \tilde{a}_3 \tilde{a}_5^2 \tilde{a}_8+5 \tilde{a}_3 (-392 \tilde{a}_6^3+135 \tilde{a}_3 \tilde{a}_7 \tilde{a}_8+105 \tilde{a}_3 \tilde{a}_6 \tilde{a}_9)))/(1763596800 \tilde{a}_3^{37/3}).
\end{aligned}
\]
\normalsize
Substituting $t=\psi(s)$ into $\gamma(t)$, and by a straightforward 
calculation, we see 
$\gamma(\psi(s))=(s^3/6,r_{3,4}s^4/4!)+O(4)$, and we have \eqref{eq:delta34}.
Under the condition $r_{3,4}=0$,
we have
$\gamma(\psi(s))=(s^3/6,r_{3,5}s^5/5!)+O(5)$, and we have \eqref{eq:delta35}.
We assume $r_{3,4}=r_{3,5}=0$,
Then we see 
$\gamma(\psi(s))=(s^3/6,\beta_{3,6}s^6/6!+\beta_{3,7}s^7/7!)+O(7)$, 
and we have \eqref{eq:beta36} and \eqref{eq:delta37}.
We assume $r_{3,7}=0$.
Then we see 
$\gamma(\psi(s))=(s^3/6,\beta_{3,6}s^6/6!+\beta_{3,8}s^8/8!)+O(8)$, 
and we have \eqref{eq:delta38}.
We assume $r_{3,8}=0$.
Then we see 
$\gamma(\psi(s))=(s^3/6,\beta_{3,6}s^6/6!+\beta_{3,9}s^9/9!
+r_{3,10}s^{10}/10!)+O(10)$,
and we have \eqref{eq:beta39} and \eqref{eq:delta310}.
\end{proof}

\begin{example}\label{ex:curvefront}
Let $\gamma$ be a curve-germ $\A^{m+1}$-equivalent to $(t^m,t^{m+1})$.
We set
\begin{equation}\label{eq:gammam}
\gamma(t)=
\left(
\sum_{i=m}^{m+1}
\dfrac{a_i}{i!}t^i,\ 
\sum_{i=m}^{m+1}
\dfrac{b_i}{i!}t^i\right)+O(m+1)\quad((a_m,b_m)\ne(0,0)).
\end{equation}
Then by a standard rotation $A$ in $\R^2$ and a parameter change
$$
t\mapsto \bar a^{-1/m}\left(
t-\dfrac{\bar a_{m+1}}{m(m+1)a^{(m+1)/m}}t^2\right),
$$
we see
$$
A\gamma(t)
=\left(\dfrac{t^m}{m!},
\dfrac{r_{m,m+1}}{(m+1)!}t^{m+1}\right),\quad
\left(
r_{m,m+1}
=
\dfrac{\bar b_{m+1}}{\bar a_m^{(m+1)/m}}\right).
$$
Thus the $(m,m+1)$-cuspidal curvature is
$r_{m,m+1}$.
Here, $\bar a_i$ and $\bar b_i$ are 
$$
\bar a_i=
\dfrac{\gamma^{(m)}(0)\cdot\gamma^{(i)}(0)}{i!|\gamma^{(m)}(0)|},\quad
\bar b_i=
\dfrac{\det(\gamma^{(m)}(0),\gamma^{(i)}(0))}{i!|\gamma^{(m)}(0)|}.
$$
\end{example}

%\section{thebibliography}


\begin{thebibliography}{9}
\bibitem{docarmo}
M. P. do Carmo, {\it Differential geometry of curves and surfaces}, 
Prentice-Hall, Inc., Englewood Cliffs, N.J., 1976. \MR{0394451}

\bibitem{dz}
{W. Domitrz and M. Zwierzy\'{n}ski, 
{\it The {G}auss-{B}onnet theorem for coherent tangent bundles over surfaces with boundary and its applications}, 
J. Geom. Anal. {\bf 30} (2020), 3243--3274. \MR{4105152}}

\bibitem{maxface} 
S. Fujimori, K. Saji, M. Umehara and K. Yamada,
{\it Singularities of maximal surfaces}. Math. Z. {\bf 259}  (2008), 827--848. \MR{2403743} 
\bibitem{fukui} 
T. Fukui,
{\it Local differential geometry of singular curves with finite multiplicities},
Saitama Math. J. {\bf 31} (2017), 79--88. \MR{3663296}

\bibitem{framed} 
T. Fukunaga, and M. Takahashi,
{\it Framed surfaces in the Euclidean space}. 
Bull. Braz. Math. Soc. (N.S.) {\bf 50} (2019), 37--65. \MR{3935057}

\bibitem{intrinsic-frontal} 
M. Hasegawa, A. Honda, K. Naokawa, K. Saji, M. Umehara and K. Yamada,
{\it Intrinsic properties of surfaces with singularities}. 
Internat. J. Math. {\bf 26} (2015), 1540008, 34 pp. \MR{3338072}

\bibitem{honda-saji} 
A. Honda and K. Saji,
{\it Geometric invariants of $5/2$-cuspidal edges}. 
Kodai Math. J. {\bf 42} (2019), 496--525. \MR{3815292}

\bibitem{honda-sato} 
A. Honda and H. Sato,
{\it Singularities of spacelike mean curvature one surfaces in de Sitter space},
prepint, arXiv:2103.13849.

\bibitem{ifrt-book}
S. Izumiya, M. C. Romero Fuster, M. A. S. Ruas and F. Tari, 
{\it Differential geometry from a singularity theory viewpoint}, 
World Scientific Publishing Co. Pte. Ltd., Hackensack, NJ, 2016. \MR{3409029}

\bibitem{ist}
S. Izumiya, K. Saji and N. Takeuchi,
{\it Flat surfaces along cuspidal edges},
J. Singul. {\bf 16} (2017), 73--100. \MR{3655304}

\bibitem{martins-saji} 
L. F. Martins and K. Saji,
{\it Geometric invariants of cuspidal edges}. 
Canad. J. Math. {\bf 68} (2016), 445--462. \MR{3484374}

\bibitem{msuy} L. F. Martins, K. Saji, M. Umehara and K. Yamada,
{\it Behavior of Gaussian curvature and
mean curvature near non-degenerate singular
points on wave fronts}, Geometry and Topology of Manifold,
 Springer Proc. Math. \& Statistics, 2016, 247--282. \MR{3555987}
 
\bibitem{msst} 	
L. F. Martins, K. Saji, S. P. Santos and K. Teramoto,
{\it Singular surfaces of revolution with prescribed unbounded mean curvature},
An. Acad. Brasil. Ci\^enc. {\bf 91} (2019), no. 3, e20170865, 10 pp. \MR{4017299}

\bibitem{front}	K. Saji, M. Umehara, and K. Yamada,
{\itshape The geometry of fronts},
Ann. of Math. {\bf 169} (2009), 491--529. \MR{2480610}
\bibitem{shibaume}
S. Shiba and M. Umehara,
{\it The behavior of curvature functions at cusps and inflection points},
Differential Geom. Appl. {\bf 30} (2012), no. 3, 285--299. \MR{2922645}

\bibitem{tt} M. Takahashi and K. Teramoto, {\it Surfaces of revolution of
frontals in the Euclidean space}, 
Bull. Brazilian Math. Soc. {\bf 51} (2020), 887--914. \MR{4167695}
\end{thebibliography}
\end{document}